\documentclass[final,11pt,english]{article}

\usepackage{amsfonts,amsmath,amsthm,amscd,amssymb,latexsym,amsbsy,pb-diagram,cite,caption2}

\usepackage{ifpdf}
\ifpdf
\usepackage[pdftex]{graphicx}
\usepackage[pdftex,dvipsnames,usenames]{color}
\else
\usepackage{graphicx}
\usepackage[dvips,dvipsnames,usenames]{color}
\fi

\usepackage[mathscr]{eucal}

\usepackage[notref,notcite]{showkeys}

\newtheorem{definition}{Definition}[section]
\newtheorem{theorem}[definition]{Theorem}
\newtheorem{lemma}[definition]{Lemma}
\newtheorem{proposition}[definition]{Proposition}
\newtheorem{corollary}[definition]{Corollary}

\theoremstyle{definition}
\newtheorem{remark}[definition]{Remark}
\newtheorem{example}[definition]{Example}

\numberwithin{equation}{section}

\DeclareMathOperator{\dist}{dist}

\begin{document}

\title{\bf Tangent metric spaces to starlike\\ sets on the plane}

\author{{\bf Oleksiy Dovgoshey:$(\boxtimes)$} \\\smallskip\\  Institute of Applied Mathematics and
Mechanics of NASU,\\ R.~Luxemburg str. 74, Donetsk 83114, Ukraine
\\ {\it aleksdov@mail.ru}\\\bigskip\\
{\bf Fahreddin Abdullayev} and {\bf Mehmet K\"{u}\c{c}\"{u}kaslan:}\\\smallskip\\
Mersin University Faculty of Literature and Science,\\ Department
of Mathematics,\\ 33342 Mersin, Turkey\\{\it fabdul@mersin.edu.tr}
and {\it mkucukaslan@mersin.edu.tr}}

\date{}

\maketitle
\begin{abstract}
Let $A\subseteq\mathbb C$ be a starlike set with a center $a$. We
prove that every tangent space to $A$ at the point $a$ is
isometric to the smallest closed cone,
 with the vertex $a$, which includes $A$. A partial converse to this result is
obtained. The tangent space to convex sets is also discussed.
{\bf 2000 Mathematic Subject Classification:} 54E35
\end{abstract}

\section{Introduction and  main results}

Analysis on metric spaces with no a priory smooth structure has
been rapidly developed resently. This development is closely
related to some generalizations of the differentiability.
Important examples of such generalizations and even an axiomatics
of so-called ``pseudo-gradients'' can be found in
\cite{Am,AmKi1,AmKi2,Ch,Ha,He,HeKo,Sh} and respectively in
\cite{Am1}. In almost all above-mentioned books and papers the
generalized differentiations involve an induced linear structure
 that makes
possible to use the classical differentiations in the linear
normed spaces. A new {\it intrinsic} approach to the introduction
of the ``smooth'' structure by means of the construction of
``tangent spaces'' for general metric spaces was proposed by
O.~Martio and by the first author of the present paper in
\cite{DM}.

In the present paper we prove that for every starlike set
$A\subseteq\mathbb C$ with a center $a$ all tangent spaces to $A$
at the point $a$ are isometric to the smallest closed cone  which
includes $A$  and has the vertex $a$. A partial converse to this
result is also obtained. Important particular cases $A=\mathbb R,\
A=\mathbb R^+$ and $A=\mathbb C$ are considered in details. The
results of the paper were partly published in the preprint from
\cite{DAK1}.

For convenience we recall the main notions from
\cite{DM,DAK1,Dov}.

Let $(X,d)$ be a metric space and let $a$ be a point of $X$. Fix a
sequence $\tilde r$ of positive real numbers $r_n$ which tend to
zero. In what follows this sequence $\tilde r$ is called a
\textit{normalizing sequence}. Let us denote by  ${\tilde X}$ the
set of all sequences of points from $X$.

\begin{definition}
\label{1:d1.1} Two sequences $\tilde{x},\tilde{y}\in \tilde X$,
$\tilde x=\{ x_n\}_{n\in \mathbb{N}}$ and $\tilde y=\{ y_n\}_{n\in
\mathbb{N}}$,  are mutually stable (with respect to a normalizing
sequence $\tilde r=\{ r_n\}_{n\in \mathbb{N}}$) if there is a
finite limit
\begin{equation}  \label{1:eq1.1}
\lim_{n\to\infty} \frac{d(x_n, y_n)}{r_n} := \tilde d_{\tilde
r}(\tilde x, \tilde y)=\tilde d(\tilde x, \tilde y).
\end{equation}
\end{definition}

We shall say that a family $\tilde F\subseteq \tilde X$ is
\textit{self-stable} (w.r.t.  $\tilde r$) if every two $\tilde
x,\tilde y \in \tilde F$ are mutually stable. A family $\tilde F
\subseteq\tilde X$ is \emph{maximal self-stable} if $\tilde{F}$ is
self-stable and for an arbitrary $\tilde z\in \tilde X$ either
$\tilde z\in \tilde F$ or there is $\tilde x\in \tilde F$ such
that $\tilde x$ and $\tilde z$ are not mutually stable.

A standard application of Zorn's Lemma leads to the following

\begin{proposition}
\label{1:p1.2} Let $(X, d)$ be a metric space and let $a\in X$.
Then for every normalizing sequence $\tilde r=\{ r_n\}_{n\in
\mathbb{N}}$ there exists a maximal self-stable family $\tilde
X_a=\tilde X_{a, \tilde r}$ such that $\tilde a:=\{ a, a,
\dots\}\in \tilde X_a$.
\end{proposition}

Note that the condition $\tilde a \in \tilde X_a$ implies the
equality
\begin{equation*}
\lim_{n\to\infty} d(x_n, a)=0
\end{equation*}
for every $\tilde x=\{ x_n\}_{n\in \mathbb{N}}\in\tilde X_a$.

Consider a function $\tilde d : \tilde X_a\times \tilde X_a \to
\mathbb{R}$ where $\tilde d(\tilde x, \tilde y)=\tilde d_{\tilde
r}(\tilde x, \tilde y)$ is defined by \eqref{1:eq1.1}. Obviously,
$\tilde d$ is symmetric and nonnegative. Moreover, the triangle
inequality for the original metric $d$ implies
\begin{equation*}
\tilde d (\tilde x, \tilde y) \leq \tilde d(\tilde x, \tilde
z)+\tilde d(\tilde z, \tilde y)
\end{equation*}
for all $\tilde x, \tilde y, \tilde z\in\tilde X_a$. Hence
$(\tilde X_a, \tilde d)$ is a pseudometric space.

\begin{definition}
\label{1:d1.3} The pretangent space to the space $X$ at the point
$a$ w.r.t.  $\tilde r$ is the metric identification of the
pseudometric space $(\tilde X_{a, \tilde r}, \tilde d)$.
\end{definition}

Since the notion of pretangent space is basic for the present
paper, we remind this metric identification construction.

Define a relation $\sim$ on $\tilde X_a$ by $\tilde x \sim \tilde
y$ if and only if $\tilde d(\tilde x, \tilde y)=0$. Then $\sim $
is an equivalence relation. Let us denote by
$\Omega_{a}=\Omega_{a, \tilde r}=\Omega_{a,\tilde r}^X$ the set of
equivalence classes in $\tilde X_a$ under the equivalence relation
$\sim$. It follows from general properties of pseudometric spaces,
see, for example, \cite[Chapter 4, Th.~15]{Kell}, that if $\rho$
is defined on $\Omega_a$ by
\begin{equation}  \label{1:eq1.2}
\rho(\alpha, \beta) :=\tilde d(\tilde x, \tilde y)
\end{equation}
where $\tilde x\in \alpha$ and $\tilde y \in \beta$, then $\rho$
is the well-defined metric on $\Omega_a$. The metric
identification of $(\tilde X_a, \tilde d)$ is, by definition, the
metric space $(\Omega_a, \rho)$.

Remark that $\Omega_{a, \tilde r} \neq\emptyset$  because the
constant sequence $\tilde a$ belongs to $\tilde X_{a,\tilde r}$,
see Proposition \ref{1:p1.2}.

Let $\{n_k\}_{k\in \mathbb{N}}$ be an infinite, strictly
increasing sequence of natural numbers. Let us denote by $\tilde
r^{\prime}$ a subsequence $\{ r_{n_k}\}_{k\in\mathbb{N}}$ of the
normalizing sequence $\tilde r=\{ r_n\}_{n\in \mathbb{N}}$ and let
$\tilde x^{\prime}:=\{ x_{n_k}\}_{k\in \mathbb{N}}$ for every
$\tilde x=\{x_n\}_{n\in \mathbb{N}}\in\tilde{X}$.  It is clear
that if $\tilde x$ and $\tilde y $ are mutually stable w.r.t.
$\tilde r$, then $\tilde x^{\prime}$ and $\tilde y^{\prime}$ are
mutually stable w.r.t. $\tilde r^{\prime}$ and
\begin{equation}  \label{1:eq1.3}
\tilde d_{\tilde r} (\tilde x, \tilde y)=\tilde d_{\tilde
r^{\prime}}( \tilde x^{\prime},\tilde y^{\prime}).
\end{equation}
If $\tilde X_{a, \tilde r}$ is a maximal self-stable (w.r.t.
$\tilde r$) family, then, by Zorn's Lemma, there exists a maximal
self-stable (w.r.t. $\tilde r^{\prime}$) family $\tilde X_{a,
\tilde r^{\prime}}$ such that
\begin{equation}\label{1.3"}
\{ \tilde x^{\prime}:\tilde x \in \tilde X_{a, \tilde
r}\}\subseteq \tilde X_{a, \tilde r^{\prime}}.
\end{equation}
Denote by $\mathrm{in}_{\tilde r^{\prime}}$ the mapping from
$\tilde X_{a, \tilde r}$ to $\tilde X_{a, \tilde r^{\prime}}$ with
$\mathrm{in}_{\tilde r^{\prime}}(\tilde x)=\tilde x^{\prime}$ for
all $\tilde x\in \tilde X_{a, \tilde r}$. If follows from
\eqref{1:eq1.3} that after the metric identifications
$\mathrm{in}_{\tilde r^{\prime}}$ passes to an isometric embedding
$\mathrm{em}^{\prime}$: $\Omega_{a, \tilde r}\to \Omega_{a, \tilde
r^{\prime}}$ under which the diagram
\begin{equation}  \label{1:eq1.4}
\begin{array}{ccc}
\tilde X_{a, \tilde r} & \xrightarrow{\ \ \mbox{in}_{\tilde r'}\ \
} &
\tilde X_{a, \tilde r^{\prime}} \\
\!\! \!\! \!\! \!\! \! p\Bigg\downarrow &  & \! \!\Bigg\downarrow
p^{\prime}
\\
\Omega_{a, \tilde r} & \xrightarrow{\ \ \mbox{em}'\ \ \ } &
\Omega_{a, \tilde r^{\prime}}
\end{array}
\end{equation}
is commutative. Here $p$ and $p^{\prime}$ are metric
identification mappings, $p(\tilde x):=\{\tilde y\in \tilde X_{a,
\tilde r}:\tilde d_{\tilde r}(\tilde x,\tilde y)=0\}$,
$p^{\prime}(\tilde x):=\{ \tilde y\in\tilde X_{a,\tilde
r^{\prime}}:\tilde d_{\tilde r^{\prime}}(\tilde x,\tilde y)=0\}$.

Let $X$ and $Y$ be two metric spaces. Recall that a map $f:X\to Y$
is called an \textit{isometry} if $f$ is distance-preserving and
onto.

\begin{definition}
\label{1:d1.4} A pretangent $\Omega_{a, \tilde r}$ is tangent if
$\mathnormal{em}^{\prime}$: $\Omega_{a, \tilde r}\to\Omega_{a,
\tilde r^{\prime}}$ is an isometry for every $\tilde r^{\prime}$.
\end{definition}
Note that the property to be tangent does not depend on the choice
of $\tilde X_{a,\tilde r'}$ in \eqref{1.3"}, see Proposition
\ref{1:p1.5} in the present paper.

Let $X$ be a metric space with a marked point $a,$ $%
\tilde{r}$ a normalizing sequence, $\tilde{X}_{a,\tilde{r}}$ a
maximal self-stable family and $\Omega _{a,\tilde{r}}$ the
corresponding pretangent space.
\begin{definition}\label{8:d2.14}
 The pretangent space
$\Omega_{a,\tilde r}$ lies in a tangent space if there is a
maximal self-stable family $\tilde{X}_{a,\tilde r'}$ such that
\eqref{1.3"} holds and if $\Omega _{a,\tilde{r}^{\prime }},$ the
metric identification of $\ \tilde{X}_{a,\tilde r'}$, is tangent.
\end{definition}
Let $(X,d)$ be a metric space  with a marked point $a$, let $Y$
and $Z$ be subspaces of $X$ such that $a\in Y\cap Z$ and let
$\tilde r=\{r_n\}_{n\in\mathbb N}$ be a normalizing sequence.
\begin{definition}\label{6:d5.1}
The subspaces $Y$ and $Z$ are {\it tangent equivalent} at the
point $a$ w.r.t.  $\tilde r$ if for every $\tilde
y_1=\{y_n^{(1)}\}_{n\in\mathbb N}\in\tilde Y$ and for every
$\tilde z_1=\{z_n^{(1)}\}_{n\in\mathbb N}\in\tilde Z$ with finite
limits
$$
\tilde d_{\tilde r}(\tilde a, \tilde
y_1)=\lim_{n\to\infty}\frac{d(y_n^{(1)},a)}{r_n}\quad\text{and}\quad
\tilde d_{\tilde r}(\tilde a, \tilde
z_1)=\lim_{n\to\infty}\frac{d(z_n^{(1)},a)}{r_n}
$$
there exist $\tilde y_2=\{y_n^{(2)}\}_{n\in\mathbb N}\in\tilde Y$
and $\tilde z_2=\{z_n^{(2)}\}_{n\in\mathbb N}\in\tilde Z$ such
that
$$
\lim_{n\to\infty}\frac{d(y_n^{(1)},z_n^{(2)})}{r_n}=\lim_{n\to\infty}
\frac{d(y_n^{(2)},z_n^{(1)})}{r_n}=0.
$$
\end{definition}
We shall say that $Y$ and $Z$ are {\it strongly tangent
equivalent} at $a$ if $Y$ and $Z$ are tangent equivalent at $a$
for all normalizing sequences $\tilde r$.

Let $A$ be a  set in a linear topological space $X$ over
$\mathbb{R}$. The set $A$ is termed starlike with a center $a$ if
$$
[a,b]=\{x\in X:x=a+t(b-a),\quad t\in[0,1]\}\subseteq A
$$
for all $b\in A$. Moreover, $A$ is a \textit{cone } with the {\it
vertex} $a$ if the ray
\begin{equation}
l_{a}(b):=\left\{ x\in X:x=a+t(b-a),t\in \mathbb{R} ^{+}\right\}
\label{8:eq22.34}
\end{equation}%
lies in $A$ for every $ b\in A$.  For  nonvoid sets
$X\subseteq\mathbb C$ and  $a\in X$ define  $Con_{a}X\ (Conv_aX)$
as the intersection of all closed (closed convex)  cones
$A\supseteq X$ with the common vertex $a$.

Now we are ready to formulate the first result of the   paper.
\begin{theorem}\label{t:1.7}
Let $X\subseteq\mathbb C$ be a set with a marked point $a$. If $X$
is  starlike  with the center $a$, then for each tangent space
$\Omega_{a,\tilde r}^X$ there is an isometry
$$
\psi:\Omega_{a,\tilde r}^X\to Con_a(X),\qquad \psi(\alpha)=a,
$$
where $\alpha=p(\tilde a)$, see \eqref{1:eq1.4}, and, moreover,
every pretangent space $\Omega_{a,\tilde r}^X$ lies in some
tangent space $\Omega_{a,\tilde r'}^X$.
\end{theorem}

This theorem can be rewritten in a slightly more general form.
\begin{theorem}\label{t:1.11}
Let $X\subseteq\mathbb C$ be a set with a marked point $a$.
Suppose that $X$ is strongly tangent equivalent (at the point $a$)
to a starlike set with the center $a$. Then all pretangent spaces
to $X$ at the point $a$ lie in tangent spaces and there is a
closed cone $B\subseteq\mathbb C$ with a vertex $b$ such that for
every tangent space $\Omega_{a,\tilde r}^X$ there exists an
isometry $\psi:\Omega^X_{a,\tilde r}\to B,\ \psi(\alpha)=b$, where
$\alpha=p(\tilde a)$, see \eqref{1:eq1.4}.
\end{theorem}

Theorem \ref{t:1.7} admits a partial converse.

Let $l=l_a(b)$ be a ray with a vertex $a$, let $X\subseteq\mathbb
C,\ a\in X$ and let $\beta>0$. Consider the two-sided angular
sector
\begin{equation}\label{1.5*}
\Gamma(a,l,\beta):=\{z\in\mathbb C:\dist(z,l)\leq\beta|z-a|\}
\end{equation}
where, as usual,
$$
\dist(z,l)=\inf_{w\in l}|z-w|.
$$
Write
\begin{equation}\label{eq1.6}
R(X,l,\beta):=\{|z-a|:z\in X\cap\Gamma(a,l,\beta)\},
\end{equation}
i.e., a positive number $t$ belongs to $R(X,l,\beta)$ if and only
if the sphere $S(a,t)=\{z\in X:|z-a|=t\}$ with the center $a$ and
the radius $t$ and the sector $\Gamma(a,l,\beta)$ have a nonvoid
intersection. In what follows we will use a porosity of the set
$R(X,l,\beta)$, so recall a definition.
\begin{definition}
\label{9:d1} Let $A\subseteq\mathbb{R}$ and let $x\in A$. The
right-side porosity of $A$ at the point $x$ is the quantity
\begin{equation}
p(A):=\limsup_{h\rightarrow 0}\frac{l(x,h,A)}{h} \label{9:eq1}
\end{equation}%
where $l(x,h,A)$ is the length of the longest interval in $\left[ x,x+h%
\right] \backslash A.$
\end{definition}
\begin{theorem}\label{t:1.9}
Let $X\subseteq\mathbb C$ be a set with a marked point $a$.
Suppose that all pretangent spaces to $X$ at the point $a$ lie in
tangent spaces and there is a closed cone $B\subseteq\mathbb C$
with a vertex $b$ such that for every tangent space
$\Omega_{a,\tilde r}^X$ there exists an isometry
$\psi:\Omega_{a,\tilde t}^X\to B,\ \psi(\alpha)=b$, where
$\alpha=p(\tilde a)$, see \eqref{1:eq1.4}.

Then for every ray $l$ with the vertex $a$ we have either
\begin{equation}\label{eq1.9}
\lim_{\beta\to0}p(R(X,l,\beta))=0 \quad \text{or} \quad
\lim_{\beta\to0}p(R(X,l,\beta))=1.
\end{equation}
 where $p(R(X,l,\beta))$ is the right-side porosity of
$R(X,l,\beta)$ at the point $0$.
\end{theorem}
Since every convex set $X$  is starlike,  Theorem \ref{t:1.7}
implies the following
\begin{corollary}
\label{8:t2.15} Let $Y$ be a convex subset of $%
\mathbb{C} $ with a marked point $a$ and let $\tilde{r}$ be a
normalizing sequence. The following statements hold for every
pretangent space $\Omega_{a,\tilde r}^Y$.

(i) If the space $\Omega_{a,\tilde{r}}^Y$  is tangent, then
$\Omega_{a,\tilde{r}}^Y$ and $Conv_{a}(Y)$ are isometric.

(ii) If $\Omega_{a,\tilde{r}}^Y$ is pretangent, then $\Omega_{a,
\tilde{r}}^Y$ lies in some tangent
$\Omega_{a,\tilde{r}^{\prime}}^Y$.
\end{corollary}

For $X=\mathbb R,\ X=\mathbb R^+=[0,\infty[$ or $X=\mathbb C$ all
pretangent spaces $\Omega_{a,\tilde r}^X$  are tangent, see
Section 3 of the present paper, but for an arbitrary convex
$X\subseteq\mathbb C$ pretangent spaces can cease to be tangent.

Recall that convex set $X$ is termed a convex body if $Int\,
X\ne\emptyset$.
\begin{proposition}\label{p:1.12}
Let $X$ be a convex body in the plane and let $a\in\partial X$.
Then for every normalizing sequence $\tilde r$ there is a maximal
self-stable family $\tilde X_{a,\tilde r}$ such that the
corresponding space $\Omega_{a,\tilde r}$ is not tangent.
\end{proposition}

\section{Auxiliary results}

In this section we collect some results related to pretangent and
tangent spaces of general metric spaces.

\begin{proposition}
\label{1:p1.5} Let $X$ be a metric space with a marked point $a$,
$ \tilde{r}$ a normalizing sequence and $\tilde{X}_{a,\tilde{r}}$
a maximal self-stable family with the corresponding pretangent
space $\Omega _{a,\tilde{r}}$. The following statements are
equivalent.

(i) $\Omega_{a,\tilde{r}}$ is tangent.

(ii) For every subsequence $\tilde{r}^{\prime }$ of the sequence $
\tilde{r}$ the family $\left\{ \tilde{x}^{\prime }:\tilde{x}\in
\tilde{X}_{a,\tilde{r}}\right\}$ is maximal self-stable w.r.t.
$\tilde{r}^{\prime}$.

(iii) A function $em^{\prime }:\Omega_{a,\tilde{r}}\longrightarrow
\Omega _{a,\tilde{r}^{\prime }}$ is surjective for every
$\tilde{r} ^{\prime}$.

(iv) A function $in_r':\,\tilde{X}_{a,\tilde r} \longrightarrow
\tilde{X}_{a,\tilde{r}^{\prime}}$ is surjective for every
$\tilde{r}'$.
\end{proposition}
For the proof see \cite[Proposition 1.2]{Dov} or \cite[Proposition
1.5]{DM}.

Let  $\tilde F\subseteq\tilde X$. For a normalizing sequence
$\tilde r$ we define a family $[\tilde F]_Y=[\tilde F]_{Y,\tilde
r}$ by the rule
\begin{equation}\label{6:eq5.1}
(\tilde y\in[\tilde F]_Y)\Leftrightarrow((\tilde y\in\tilde
Y)\&(\exists\,\tilde x\in\tilde F:\tilde d_{\tilde r}(\tilde
x,\tilde y)=0)).
\end{equation}

\begin{proposition}[\!\!\cite{Dov}]\label{6:p5.2}
Let $Y$ and $Z$ be subspaces of a metric space $X$ and let $\tilde
r$ be a normalizing sequence. Suppose that $Y$ and $Z$ are tangent
equivalent (w.r.t. $\tilde r$) at a point $a\in Y\cap Z$. Then
following statements hold for every maximal self-stable (in
$\tilde Z$) family $\tilde Z_{a,\tilde r}$.
\begin{itemize}
\item[$(i)$] The family $[\tilde Z_{a,\tilde r}]_Y$ is maximal
self-stable (in $\tilde Y$) and we have the equalities
\begin{equation}\label{6:eq5.2}
[[\tilde Z_{a,\tilde r}]_Y]_Z=\tilde Z_{a,\tilde r}=[\tilde
Z_{a,\tilde r}]_Z.
\end{equation}

\item[$(ii)$] If $\Omega^Z_{a,\tilde r}$ and $\Omega^Y_{a,\tilde
r}$ are metric identifications of $\tilde Z_{a,\tilde r}$ and,
respectively, of $\tilde Y_{a,\tilde r}:=[\tilde Z_{a,\tilde
r}]_Y$, then the mapping
\begin{equation}\label{6:eq5.3}
\Omega_{a,\tilde
r}^Z\ni\alpha\longmapsto[\alpha]_Y\in\Omega_{a,\tilde r}^Y
\end{equation}
is an isometry. Furthermore, if $\Omega_{a,\tilde r}^Z$ is
tangent, then $\Omega^Y_{a,\tilde r}$ also is tangent.
\end{itemize}
\end{proposition}

The following lemma is a partial generalization of
Proposition~\ref{6:p5.2} (i).
\begin{lemma}\label{l:2.3}
Let $Z$ and $Y$ be subspaces of a metric space $(X,d)$, $a\in
X\cap Y$, $\tilde r$ a normalizing sequence, $\tilde Z_{a,\tilde
r}$ and $\tilde Y_{a,\tilde r}$ maximal self-stable families such
that
\begin{equation}\label{eq2.4}
\tilde Y_{a,\tilde r}=[\tilde Z_{a,\tilde r}]_{Y,\tilde r}.
\end{equation}
Suppose  $Y$  and $Z$ are strongly tangent equivalent at the point
$a$. Then the equality
\begin{equation}\label{eq2.5}
\{\tilde z':\tilde z\in\tilde Z_{a,\tilde r}\}=[\{\tilde
y':y\in\tilde Y_{a,\tilde r}\}]_{Z,\tilde r'}
\end{equation}
holds for every subsequence $\tilde r'$ of the sequence $\tilde
r$.
\end{lemma}
\begin{proof}
Let $\tilde r'=\{r_{n_k}\}_{k\in\mathbb N}$ be a subsequence of
$\tilde r=\{r_n\}_{n\in\mathbb N}$. We first note that
\eqref{eq2.4} and \eqref{6:eq5.2} imply the equality
$$
\tilde Z_{a,r}=[Y_{a,\tilde r}]_{Z,\tilde r}.
$$
Consequently, if $\tilde z'=\{z_{n_k}\}_{k\in\mathbb N}$ belongs
to the set in the left-hand side of \eqref{eq2.5}, then there is
$\tilde y=\{y_n\}\in\tilde Y_{a,\tilde r}$ such that
$$
\lim_{n\to\infty}\frac{d(y_n,z_n)}{r_n}=0.
$$
Hence
\begin{equation}\label{eq2.6}
\lim_{k\to\infty}\frac{d(y_{n_k},z_{n_k})}{r_{n_k}}=0.
\end{equation}
The last equality means that $\tilde z'$ belongs to the set in the
right-hand side of \eqref{eq2.5}. Conversely, if
$$
\tilde z'=\{z_{n_k}\}_{k\in\mathbb N}\in[\{\tilde y':\tilde
y\in\tilde Y_{a,\tilde r}\}]_{Z,\tilde r'},
$$
then \eqref{eq2.6} holds with some $\tilde y=\{y_n\}\in\tilde
Y_{a,\tilde r}$. Hence, by \eqref{eq2.4}, there is $\tilde
z_1=\{\tilde z_n^{(1)}\}_{n\in\mathbb N}$ such that
\begin{equation}\label{eq2.7}
\lim_{n\to\infty}\frac{d(y_n,z_n^{(1)})}{r_n}=0.
\end{equation}
Let us define $\tilde z_2=\{z_n^{(2)}\}_{n\in\mathbb N}\in\tilde
Z$ by the rule
\begin{equation}\label{eq2.8}
z_n^{(2)}:=\begin{cases} z_{n_k}&\text{if $n=n_k$ for some $k$}\\
z_n^{(1)}&\text{otherwise.}
\end{cases}
\end{equation}
Limit relations \eqref{eq2.6} and \eqref{eq2.7} imply that $\tilde
d_{\tilde r}(\tilde z_1,\tilde z_2)=0$. Moreover, by the second
equality in \eqref{6:eq5.2}, we have $\tilde z_2\in Z_{a,\tilde
r}$. Consequently, by \eqref{eq2.8}, we have
$$
\tilde z'=\tilde z_2'\in\{\tilde z':\tilde z\in\tilde Z_{a,\tilde
r}\}.
$$
\end{proof}
\begin{proposition}\label{p:2.4}
Let $Y$ and $Z$ be subspaces of a metric space $X$ and let $a$ be
a point in $Y\cap Z$. Suppose that $Y$ and $Z$ are strongly
tangent equivalent at the point $a$ and that each  pretangent
space $\Omega_{a,\tilde r}^Y$ lies in some tangent space
$\Omega_{a,\tilde r'}^Y$. Then each pretangent space
$\Omega_{a,\tilde r}^Z$ lies in some tangent space
$\Omega_{a,\tilde r'}^Z$.
\end{proposition}
\begin{proof}
Let $\Omega_{a,\tilde r}^Z$ be a pretangent space to $Z$ at the
point $a$ and let $\tilde Z_{a,\tilde r}$ be the corresponding
maximal (in $\tilde Z$), self-stable family. Write
$$
\tilde Y_{a,\tilde r}:=[\tilde Z_{a,\tilde r}]_{Y,\tilde r}.
$$
Then, by Proposition \ref{6:p5.2} (i), $\tilde Y_{a,\tilde r}$ is
maximal (in $\tilde Y$), self-stable family and, by the
supposition, there are a subsequence $\tilde r' $ of $\tilde r$
and a maximal self-stable family $\tilde Y_{a,\tilde r'}$ such
that
\begin{equation}\label{eq2.9}
\{\tilde y':\tilde y\in\tilde Y_{a,\tilde r}\}\subseteq\tilde
Y_{a,\tilde r'}
\end{equation}
and $\Omega_{a,\tilde r'}^Y$, the metric identification of $\tilde
Y_{a,\tilde r'}$, is tangent. For this $\tilde r'$ consider the
family $ \{\tilde z':\tilde z\in\tilde Z_{a,\tilde r}\}. $ By
Lemma \ref{l:2.3} we have  equality \eqref{eq2.5}. Write $ \tilde
Z_{a,\tilde r'}:=[\tilde Y_{a,\tilde r'}]_{Z,\tilde r'}. $ It
follows from \eqref{eq2.9} and \eqref{eq2.5} that $ \{\tilde
z':\tilde z\in\tilde Z_{a,\tilde r}\}\subseteq\tilde Z_{a,\tilde
r'} $ and, moreover, Proposition \ref{6:p5.2} implies that
$\Omega_{a,\tilde r'}^Z$, the metric identification of $\tilde
Z_{a,\tilde z'}$, is tangent.
\end{proof}
Let $Y$ be a subspace of a metric space $(X,d)$. For $a\in Y$ and
$t>0$ we denote by
$$
S_t^Y=S^Y(a,t):=\{y\in Y:d(a,y)=t\}
$$
the sphere (in the subspace $Y$) with the center $a$ and the
radius $t$. Similarly for $a\in Z\subseteq X$ and $t>0$ define
$$
S_t^Z=S^Z(a,t):=\{z\in Z:d(a,z)=t\}.
$$
Write
\begin{equation*}\label{6:eq5.6}
\varepsilon_a(t,Z,Y):=\sup_{z\in S_t^Z}\inf_{y\in Y}d(z,y) \quad
\text{and} \quad
\varepsilon_a(t):=\varepsilon_a(t,Z,Y)\vee\varepsilon_a(t,Y,Z).
\end{equation*}

\begin{proposition}[\!\!\cite{DM,Dov}]\label{6:t5.4}
Let $Y$ and $Z$ be subspaces of a metric space $(X,d)$ and let
$a\in Y\cap Z$. Then $Y$ and $Z$ are strongly tangent equivalent
at the point $a$ if and only if the equality
\begin{equation}\label{6:eq5.8}
\lim_{t\to0}\frac{\varepsilon_a(t)}t=0
\end{equation}
holds.
\end{proposition}
\begin{corollary}\label{c:2.5*}
Let $Y$ be a dense subset of a metric space $X$. Then $X$ and $Y$
are strongly tangent equivalent at every point $a\in Y$.
\end{corollary}
\begin{lemma}\label{p:2.6}
Let $(X,d)$ be a metric space with a marked point $a$, $\tilde
r=\{r_n\}_{n\in\mathbb N}$ a normalizing sequence and let $\tilde
X_{a,\tilde r}$ be a maximal self-stable family. Then for every
$\varepsilon>0$ and every $\tilde x=\{x_n\}\in\tilde X_{a,\tilde
r}$ with $\tilde d_{\tilde r}(\tilde x,\tilde a)>0$ there is
$n_0\in\mathbb N$ such that the double inequality
\begin{equation}\label{*2.10}
(1-\varepsilon)\tilde d_{\tilde r}(\tilde a,\tilde
x)<\frac{d(a,x_n)}{r_n}<(1+\varepsilon)\tilde d_{\tilde r}(\tilde
a,\tilde x)
\end{equation}
holds for all natural numbers $n\geq n_0$.
\end{lemma}
A simple proof is omitted here.

\begin{lemma}
\label{4:l2.8} Let $(X,d)$ be a metric space with a marked point
$a$,  $\tilde{r}=\left\{r_{n}\right\}_{n\in\mathbb{N}}$ a
normalizing sequence and $\tilde{X}_{a,\tilde{r}}$ a maximal
self-stable family and $\tilde{f}=\left\{
f_{n}\right\}_{n\in\mathbb{N}}$ a sequence of isometries
$f_{n}:X\rightarrow X$ \ with $f_{n}(a)=a$ for all
$n\in\mathbb{N}$. Then the family
\begin{equation}
\tilde{f}(\tilde{X}_{a,\tilde{r}}):=\left\{\left\{
f_{n}(x_{n})\right\} _{n\in\mathbb{N}}:\left\{ x_{n}\right\}_{n\in
\mathbb{N}}\in \tilde{X}_{a,\tilde{r}}\right\} \label{equ2.20}
\end{equation}%
is a maximal self-stable family and, in addition, the pseudometric
spaces $(\tilde{X}_{a,\tilde{r}},\tilde{d})$, $(\tilde{f}(
\tilde{X}_{a,\tilde{r}}),\tilde{d})$  are isometric. Moreover,
$\Omega_{a,\tilde r}$ and $\Omega_{a,\tilde r}^{\tilde f}$, metric
identifications of $\tilde X _{a,\tilde r}$ and, respectively, of
$\tilde f(\tilde X _{a,\tilde r})$, are simultaneously tangent or
not.
\end{lemma}

\begin{proof}
Since
\begin{equation*}
\tilde{d}(\tilde{x},\tilde{y})=\underset{n\rightarrow \infty }{
\lim }\frac{d(x_{n},y_{n})}{r_{n}}=\underset{n\rightarrow \infty
}{\lim} \frac{d(f_{n}(x_{n}),f_{n}(y_{n}))}{r_{n}},
\end{equation*}%
every two $\left\{ f_{n}(x_{n})\right\} _{n\in\mathbb{N}}$ and
$\left\{f_{n}(y_{n})\right\}_{n\in\mathbb{N}}$ are mutually stable
if $\tilde{x}=\left\{x_{n}\right\}_{n\in\mathbb{N}}$ and
$\tilde{y}=\left\{ y_{n}\right\}_{n\in\mathbb{N}}$ are mutually
stable and the mapping
\begin{equation*}
\tilde{X}_{a,\tilde{r}}\ni
\tilde{x}=\left\{x_{n}\right\}_{n\in\mathbb{N}}\longmapsto \left\{
f_{n}(x_{n})\right\}_{n\in\mathbb{N}}:=\tilde f(\tilde{x})\in
\tilde{f}(\tilde{X}_{a,\tilde r})
\end{equation*}%
is an isometry. It is clear that $\tilde{f}(\tilde{a}%
)=(a,a,...,a,...)\in \tilde{f}(\tilde{X}_{a,\tilde{r}})$. Hence it
suffices to show that $\tilde{f}(\tilde{X}_{a,\tilde r})$ is
maximal self-stable. Suppose there is
$\tilde{z}=\left\{z_{n}\right\}_{n\in\mathbb{N}}\in \tilde{X}$
such that $\tilde{z}\notin \tilde{f}(\tilde{X}_{a,\tilde{r}})$ but
$\tilde{z}$ and $\tilde{x}$ are mutually stable for all $\tilde
x\in\tilde{f}(\tilde{X}_{a,\tilde r})$. It is easy to see that
\begin{equation*}
\tilde{f}^{-1}(\tilde{z}):=\left\{
f_{n}^{-1}(z_{n})\right\}_{n\in\mathbb{N}}\notin
\tilde{X}_{a,\tilde{r}}
\end{equation*}%
where $f_{n}^{-1}$ is the inverse isometry of the isometry $f_{n}$
and that $\tilde{x}$ and $\tilde{f}^{-1}(\tilde{z})$ are mutually
stable for each $\tilde{x}\in \tilde{X}_{a,\tilde{r}}$. Hence $
\tilde{X}_{a,\tilde{r}}$ is not a maximal self-stable family,
contrary to the condition of the lemma.

 Suppose that $\Omega
_{a,\tilde{r}}$ is tangent. Let $\tilde{r}^{\prime }=\left\{
r_{n_{k}}\right\}_{n\in\mathbb{N}}$ be a subsequence of
$\tilde{r}.$ Then, by Proposition \ref{1:p1.5}, the family
$
\left\{ \left\{ x_{n_{k}}\right\} _{k\in\mathbb{N}}:\left\{
x_{n}\right\}_{n\in\mathbb{N}}\in \tilde{X}_{a,\tilde{r}}\right\}
$
is maximal self-stable. The first part of the proof implies that
the family
$
\left\{ \left\{ f_{n_{k}}(x_{n_{k}})\right\}
_{k\in\mathbb{N}}:\left\{ x_{n}\right\} _{n\in\mathbb{N}}\in
\tilde{X}_{a,\tilde{r}}\right\}
$
is also maximal self-stable. Consequently, by Proposition
\ref{1:p1.5}, $\Omega _{a,\tilde{r}}^{\tilde f}$ is tangent.
\end{proof}

\section{Tangent spaces to some model metric spaces}

In this section we describe tangent spaces to $\mathbb R^+,\
\mathbb R$ and $\mathbb C$.
\begin{example}
\label{4:e2.1} Let $X=\mathbb R$ or $X=\mathbb{R}^{+}=\left[
0,\infty \right[$ and let
$
d(x,y)=\left\vert x-y\right\vert.
$
 We claim that each
pretangent space $\Omega_{0,\tilde r}^X$ (to $X$ at the point $0$)
is tangent and isometric to $(X,d)$ for all normalizing sequences
$\tilde r$.
\end{example}

Consider the more difficult case $X=\mathbb R$.

\begin{proposition}
\label{4:p2.4} Let $\tilde{X}_{0,\tilde r}$ be maximal self-stable
family and let $\tilde{b}=\left\{b_{n}\right\}_{n\in\mathbb{N}}$
be an element of $\tilde{X}_{0,\widetilde{r}}$ such that
\begin{equation*}
\tilde{d}_{\tilde{r}}(\tilde{0},\tilde{b})=\underset{n\rightarrow
\infty }{\lim }\frac{\left\vert b_{n}\right\vert }{r_{n}}\neq 0.
\end{equation*}

The following statements are true.

(i) For every $\widetilde{y}=\left\{
y_{n}\right\}_{n\in\mathbb{N}}\in \tilde{X}_{0,\tilde{r}}$ there
is a finite limit $\underset{n\rightarrow \infty}{\lim
}\frac{y_{n}}{b_{n}}$ and, conversely, if $\tilde{y}\in \tilde{X}$
and this limit is finite, then $\tilde{y}\in
\tilde{X}_{0,\tilde{r}}$.

(ii) For every two $\tilde{x}=\left\{ x_{n}\right\}
_{n\in\mathbb{N}}$ and $\tilde{y}=\left\{
y_{n}\right\}_{n\in\mathbb{N}}$ from $\tilde{X}_{0,\tilde{r}}$ the
equality $\tilde{d}_{\tilde{r}}(\tilde{x},\tilde{y})=0$ holds if
and only if
\begin{equation*}
\underset{n\rightarrow \infty}{\lim}\frac{x_{n}}{b_{n}}=\underset{
n\rightarrow \infty}{\lim }\frac{y_{n}}{b_{n}}.
\end{equation*}

(iii) The pretangent space $\Omega _{0,\tilde{r}}$ which
corresponds to $\tilde{X}_{0,\tilde{r}}$ is isometric to
$(\mathbb{R},\left\vert .,.\right\vert)$ and tangent.
\end{proposition}

\begin{proof}
(i) If $\tilde{y}=\left\{ y_{n}\right\}_{n\in\mathbb{N}}\in
\tilde{X}_{0,\tilde{r}}$, then there are finite limits
\begin{equation*}
\tilde{d}(\tilde{0},\tilde{y})=\underset{n\rightarrow \infty }{
\lim }\frac{\left\vert y_{n}\right\vert }{r_{n}}\text{ and
}\tilde{d}( \tilde{b},\tilde{y})=\underset{n\rightarrow \infty
}{\lim }\frac{ \left\vert y_{n}-b_{n}\right\vert }{r_{n}}.
\end{equation*}%
For the case where $\tilde{d}(\tilde{0},\tilde{y})=0$ we obtain
\begin{equation*}
0=\frac{\tilde{d}(\tilde{0},\tilde{y})}{\tilde{d}(\tilde{
0},\tilde{b})}=\underset{n\rightarrow \infty }{\lim
}\frac{\left\vert y_{n}\right\vert }{\left\vert b_{n}\right\vert
}=\underset{n\rightarrow \infty}{\lim}\frac{y_{n}}{b_{n}}
\end{equation*}
because $\tilde{d}(\tilde{0},\tilde{b})\neq 0$. Suppose
$\tilde{d}(\tilde{0},\tilde{y})\neq 0,$ then we have
\begin{equation}\label{4:eq2.3}
0<\underset{n\rightarrow \infty }{\lim }\frac{\left\vert
y_{n}\right\vert}{\left\vert b_{n}\right\vert
}=\frac{\tilde{d}(\tilde{0},\tilde{y}
)}{\tilde{d}(\tilde{0},\tilde{b})}<\infty.
\end{equation}%
Write for every $t\in\mathbb{R}$
\begin{equation*}
t=\left\vert t\right\vert sgn(t)
\end{equation*}%
where, as usual,
\begin{equation*}
sgn(t)=\left\{
\begin{array}{ccc}
1 & \text{if} & t>0 \\
0 & \text{if} & t=0 \\
-1 & \text{if} & t<0.%
\end{array}%
\right.
\end{equation*}
Then, it follows from \eqref{4:eq2.3}, the limit
$\underset{n\rightarrow \infty }{\lim }\frac{y_{n}}{b_{n}}$ exists
if and only if there is the limit $\underset{n\rightarrow \infty
}{\lim }\frac{sgn(y_{n})}{sgn(b_{n})}$. If the last limit does not
exist, then there are two infinite sequences $\tilde{n}=\left\{
n_{k}\right\} _{k\in\mathbb{N}}$ and $\tilde{m}=\left\{
m_{k}\right\} _{k\in\mathbb{N}}$ of natural numbers such that
\begin{equation*}
sgn(y_{n_{k}})= sgn(b_{n_{k}})\quad  \text{ and }\quad
sgn(y_{m_{k}})= sgn(b_{m_{k}}).
\end{equation*}%
for all $k\in\mathbb N$. Consequently, we obtain
\begin{equation*}
\tilde{d}(\tilde{y},\tilde{b})=\underset{k\rightarrow \infty }{
\lim }\frac{\left\vert y_{n_{k}}-b_{n_{k}}\right\vert
}{r_{n_{k}}}=\lim_{ k\rightarrow \infty
}\frac{||y_{n_k}|-|b_{n_k}||}{r_{n_k}}
 =\left\vert
\tilde{d}(\tilde{0},\tilde{y})-\tilde{d}(
\tilde{0},\tilde{b})\right\vert
\end{equation*}
and similarly we have
\begin{equation*}
\tilde{d}(\tilde{y},\tilde{b})=\underset{k\rightarrow \infty}{\lim
}\frac{\left\vert y_{m_{k}}-b_{m_{k}}\right\vert}{r_{m_{k}}} =
\tilde{d}(\tilde{0},\tilde{y})+\tilde{d}( \tilde{0},\tilde{b}).
\end{equation*}%
Thus we have the equality%
\begin{equation*}
\tilde{d}(\tilde{0},\tilde{y})+\tilde{d}( \tilde{0},\tilde{b})
=\left\vert \tilde{d}(\tilde{
0},\tilde{y})-\tilde{d}(\tilde{0},\tilde{b})\right\vert
\end{equation*}%
which implies that
\begin{equation*}
\tilde{d}(\tilde{0},\tilde{y})\wedge \tilde{d}(\tilde{0},
\tilde{b})=0.
\end{equation*}%
It is shown that for every $\tilde{y}\in \tilde{X}_{0,\tilde{r}}$
there is a finite limit $\underset{n\rightarrow
\infty}{\lim}\frac{y_{n}}{b_{n}}$. Conversely, let
$\tilde{y}\in\tilde{X}$ and
\begin{equation}
\underset{n\rightarrow \infty }{\lim
}\frac{y_{n}}{b_{n}}=c\in\mathbb{R}.  \label{4:eq2.4}
\end{equation}%
We must   show that for every $\tilde{x}\in \tilde{X}_{0,
\tilde{r}}$ there is a finite limit $\underset{k\rightarrow \infty
}{ \lim}\frac{\left\vert y_{n}-b_{n}\right\vert}{r_{n}}$, i.e.
$\tilde{x}$ and $\tilde{y}$ are mutually stable w.r.t.
$\tilde{r}$. Since $\tilde{x} \in \tilde{X}_{0,\tilde{r}}$, we
have a finite limit
\begin{equation}
\underset{n\rightarrow \infty }{\lim
}\frac{x_{n}}{b_{n}}=k\in\mathbb{R}.  \label{4:eq2.5}
\end{equation}
Hence
\begin{equation}
\underset{n\rightarrow \infty }{\lim }\frac{\left\vert
y_{n}-x_{n}\right\vert }{\left\vert r_{n}\right\vert }=\underset{
n\rightarrow \infty }{\lim }\frac{\left\vert b_{n}\right\vert
}{r_{n}} \left\vert
\frac{x_{n}}{b_{n}}-\frac{y_{n}}{b_{n}}\right\vert =\widetilde{d}(
\tilde{0},\tilde{b})\left\vert c-k\right\vert \label{4:eq2.6}
\end{equation}%
where constants $c,k$ are defined by \eqref{4:eq2.4} and,
respectively, by \eqref{4:eq2.5}.

(ii) Statement (ii) follows from \eqref{4:eq2.6}.

(iii) Statement (i) implies that the sequence $\tilde{r}^{\ast
}=\left\{ r_{n}^{\ast }\right\} _{n\in\mathbb{N}}$ with
$$
r_n^{\ast}=r_n sgn(b_{n}),\qquad n\in\mathbb{N},
$$
belongs to
$\tilde{X}_{0,\tilde{r}}$. If we take $\tilde{r}^{\ast }$ instead
$\tilde{b}$ in \eqref{4:eq2.4} and \eqref{4:eq2.5}, then we obtain
the mapping $f:\tilde{X}_{0,\tilde{r}}\rightarrow\mathbb{R}$ where
\begin{equation*}
f(\tilde{x})=\underset{n\rightarrow \infty }{\lim }\frac{x_{n}}{
r_{n}^{\ast }},~~\tilde{x}=\left\{ x_{n}\right\}
_{n\in\mathbb{N}}\in \tilde{X}_{0,\tilde{r}}.
\end{equation*}%
It is easy to see that there is a unique mapping $\psi:$ $\Omega
_{0,\tilde{r}}\rightarrow\mathbb{R}$ such that the diagram

\begin{equation}  \label{4:eq2.7}
\begin{diagram}
\node{\tilde X_{0,\tilde r}} \arrow[2]{e,t}{p} \arrow{ese,b}{f}
                             \node[2]{\Omega_{0,\tilde r}}\arrow{s,r}{\psi}
                             \\\node[3]{\mathbb R}
\end{diagram}
%
\end{equation}
is commutative, where $p$ is the metric identification mapping,
see \eqref{1:eq1.4}. Relations \eqref{4:eq2.4}--\eqref{4:eq2.6}
imply that $\psi $ is an isometry. It remains to prove that
$\Omega _{a,\tilde{r}}$ is tangent. Let
$\tilde{n}=\left\{n_{k}\right\}_{k\in\mathbb{N}}$ be a strictly
increasing, infinite sequence of natural numbers and let $
\tilde{r}^{\prime }=\left\{ r_{n_k}\right\}_{k\in\mathbb{N}}$ be
the corresponding subsequence of the normalizing sequence
$\tilde{r}$. If $\tilde X_{0,\tilde r'}$ is a maximal self-stable
family such that
$$
\tilde X_{0,\tilde r'}\supseteq\{\tilde x':\tilde x\in\tilde
X_{0,\tilde r}\},
$$
then, by Statement (i), for every $\tilde x=\{x_k\}_{k\in\mathbb
N}\in\tilde X_{0,\tilde r}$ there is a finite limit
\begin{equation*}
\underset{k\rightarrow \infty }{\lim
}\frac{x_{k}}{r_{n_{k}}sgn(b_{n_k})}:=p.
\end{equation*}%
Define  $\tilde{y}=\left\{y_{n}\right\}_{n\in\mathbb{N}}\in
\tilde{X}$ \ by the rule
\begin{equation*}
y_{n}:=
\begin{cases}
x_{k}& \text{if there is  }n_{k} \text{ such that }n_{k}=n, \\
r_{n}sgn(b_{n})& \text{otherwise}.
\end{cases}
\end{equation*}
A simple calculation shows that
\begin{equation*}
\underset{n\rightarrow \infty }{\lim
}\frac{y_{n}}{r_{n}sgn(b_n)}=p.
\end{equation*}%
Hence, by Statement (i),  $\tilde{y}$ belongs to
$\tilde{X}_{0,\tilde{r}}$. Using Proposition \ref{1:p1.5} we see
that $\Omega _{0,\tilde{r}}$ is tangent.
\end{proof}

\begin{example}
\label{4:e2.5} Let $X=\mathbb{C}$ be the set of all complex
numbers with the usual metric $\left\vert .,.\right\vert $ and the
marked point $0$ and let $\tilde{r}=\left\{ r_{n}\right\}
_{n\in\mathbb{N}}$ be a normalizing sequence.
\end{example}

\begin{proposition} \label{4:p2.6} Let
$\tilde{X}_{0,\tilde{r}}$ be a maximal
self-stable family with the corresponding pretangent space $\Omega _{0,%
\tilde{r}}$. Then $\Omega _{0,\tilde{r}}$ is tangent and isometric
to $\mathbb{C}$.
\end{proposition}

The proof is divided into four lemmas.

\begin{lemma}
\label{4:l2.7} Let $\tilde{x}=\left\{ x_{n}\right\}
_{n\in\mathbb{N}}$ and $\tilde{y}=\left\{ y_{n}\right\} _{n\in
\mathbb{N}}$ be elements of $\tilde{X}_{0,\tilde{r}}$ such that
\begin{equation}
2\max\left\{\tilde{d}(\tilde{0},\tilde{x}),\tilde{d}(
\tilde{0},\tilde{y}),\tilde{d}(\tilde{x},\tilde{y}
)\right\} <\tilde{d}(\tilde{0},\tilde{x})+\tilde{d}(%
\tilde{0},\tilde{y})+\tilde{d}(\tilde{x},\tilde{y}).
\label{4:eq22.8}
\end{equation}%
Then following statements are equivalent for every
$\tilde{z}=\left\{ z_{n}\right\}_{n\in\mathbb{N}}\in
\widetilde{X}:$

(a) $\tilde{z}$ belongs to $\tilde{X}_{0,\tilde{r}};$

(b) There are finite limits
\begin{equation}
\tilde{d}(\tilde{0},\tilde{z})=\underset{n\rightarrow \infty }{%
\lim }\frac{\left\vert z_{n}\right\vert }{r_{n}},~~\tilde{d}(\tilde{x%
},\tilde{z})=\underset{n\rightarrow \infty }{\lim
}\frac{\left\vert
x_{n}-z_{n}\right\vert }{r_{n}}~\text{and}~\ \tilde{d}(\tilde{y},\tilde{z}%
)=\underset{n\rightarrow \infty }{\lim }\frac{\left\vert
y_{n}-z_{n}\right\vert }{r_{n}}.  \label{4:eq22.9}
\end{equation}%
\end{lemma}
\begin{proof}
The implication (a)$\Longrightarrow $(b) is trivial. Suppose that
(b) holds with $\tilde{z}=\tilde{z}_{1}=\left\{
z_{n}^{(1)}\right\} _{n\in\mathbb{N}}$  and with
$\tilde{z}=\tilde{z}_{2}=\left\{ z_{n}^{(2)}\right\}
_{n\in\mathbb{N}}$. We must  prove that there is a finite limit
\begin{equation}
\tilde{d}(\tilde{z}_{1},\tilde{z}_{2})=\underset{n\rightarrow
\infty }{\lim }\frac{\left\vert z_{n}^{(1)}-z_{n}^{(2)}\right\vert
}{r_{n}}. \label{4:eq2.10}
\end{equation}%
Write%
\begin{equation}
x_{n}:=\left\vert x_{n}\right\vert e^{i\beta
_{n}},~y_{n}:=\left\vert y_{n}\right\vert e^{i\theta _{n}},~\
z_{n}^{(1)}:=\left\vert z_{n}^{(1)}\right\vert e^{i\gamma
_{n}^{(1)}},~\ z_{n}^{(2)}:=\left\vert z_{n}^{(2)}\right\vert
e^{i\gamma _{n}^{(2)}}. \label{4:eq22.11}
\end{equation}%
Since $\tilde{x},\tilde{y}\in \tilde{X}_{0,\tilde{r}}$ and the
first relation in \eqref{4:eq22.9} holds,  there are finite limits
\begin{equation*}
R_{x}:=\underset{n\rightarrow \infty }{\lim }\frac{\left\vert
x_{n}\right\vert }{r_{n}},~R_{y}:=\underset{n\rightarrow \infty }{\lim }%
\frac{\left\vert y_{n}\right\vert
}{r_{n}},~R_{1,z}:=\underset{n\rightarrow
\infty }{\lim }\frac{\left\vert z_{n}^{(1)}\right\vert }{r_{n}},~~R_{2,z}:=%
\underset{n\rightarrow \infty }{\lim }\frac{\left\vert
z_{n}^{(2)}\right\vert }{r_{n}}.
\end{equation*}%
Consequently we have the limit relations%
\begin{equation*}
\tilde{d}(\tilde{x},\tilde{y})=\underset{n\rightarrow \infty }{%
\lim }\left\vert R_{x}e^{i\beta _{n}}-R_{y}e^{i\theta _{n}}\right\vert =%
\underset{n\rightarrow \infty }{\lim }\left\vert
R_{x}-R_{y}e^{i(\theta _{n}-\beta _{n})}\right\vert ,
\end{equation*}%
\begin{equation*}
\tilde{d}(\tilde{z}_1,\tilde{y})=\underset{n\rightarrow \infty }{%
\lim }\left\vert R_{1,z}e^{i(\gamma _{n}^{(1)}-\beta
_{n})}-R_{y}e^{i(\theta _{n}-\beta _{n})}\right\vert,
\end{equation*}
\begin{equation*}
\tilde{d}(\tilde{z}_2,\tilde{y})=\underset{n\rightarrow \infty }{%
\lim }\left\vert R_{2,z}e^{i(\gamma _{n}^{(2)}-\beta
_{n})}-R_{y}e^{i(\theta _{n}-\beta _{n})}\right\vert ,
\end{equation*}%
\begin{equation}
\tilde{d}(\tilde{x},\tilde{z}_{1})=\underset{n\rightarrow \infty
}{\lim }\left\vert R_{1,z}e^{i(\gamma _{n}^{(1)}-\beta
_{n})}-R_{x}\right\vert,\,\,
\tilde{d}(\tilde{x},\tilde{z}_{2})=\underset{n\rightarrow \infty
}{\lim }\left\vert R_{2,z}e^{i(\gamma _{n}^{(2)}-\beta
_{n})}-R_{x}\right\vert  \label{4:eq22.12}
\end{equation}%
and must prove the existence of
\begin{equation}
\tilde{d}(\tilde{z}_{1},\tilde{z}_{2})=\underset{n\rightarrow
\infty }{\lim }\left\vert R_{2,z}e^{i(\gamma _{n}^{(2)}-\beta
_{n})}-R_{1,z}e^{i(\gamma _{n}^{(2)}-\beta _{n})}\right\vert .
\label{4:eq22.13}
\end{equation}%
It is clear from \eqref{4:eq22.12}, \eqref{4:eq22.13} that,
without loss of generality, it is sufficient to take $\beta
_{n}=0$ for all $n\in \mathbb{N} $.

Moreover, \eqref{4:eq22.13} evidently holds if $R_{1,z}\cdot
R_{2,z}=0.$ Hence we may also put
\begin{equation}
R_{1,z}\neq 0\neq R_{2,z}.  \label{4:eq2.14}
\end{equation}%
Note  that \eqref{4:eq22.8} implies%
\begin{equation}
R_{x}\neq 0\neq R_{y}.  \label{4:eq22.15}
\end{equation}%
Since%
\begin{equation*}
\left\vert R_{x}-R_{y}e^{i\theta _{n}}\right\vert
^{2}=R_{x}^{2}+R_{y}^{2}-2R_{x}R_{y}\cos (\theta _{n}),
\end{equation*}%
the first relation in \eqref{4:eq22.12} and \eqref{4:eq22.15}
imply that there exists $\underset{n\rightarrow \infty }{\lim
}\cos (\theta _{n})$ and, in addition, if follows from
\eqref{4:eq22.8} that
\begin{equation}\label{4:eq22.16}
\underset{n\rightarrow \infty }{\lim }\cos (\theta _{n})\neq \pm
1,
\end{equation}%
see Remark \ref{4:r2.8} below. Similarly using \eqref{4:eq2.14},
\eqref{4:eq22.15} and last two relations from \eqref{4:eq22.12} we
see that there are limits
\begin{equation}
\underset{n\rightarrow \infty }{\lim }\cos (\gamma _{n}^{(1)})\text{ and }%
\underset{n\rightarrow \infty }{\lim }\cos (\gamma _{n}^{(2)}).
\label{4:eq22.17}
\end{equation}%
The remaining relations from \eqref{4:eq22.12} imply the existence
of
\begin{equation}
\underset{n\rightarrow \infty }{\lim }\cos (\gamma _{n}^{(1)}-\theta _{n})%
\text{ \ and }\underset{n\rightarrow \infty }{\lim }\cos (\gamma
_{n}^{(2)}-\theta _{n}).  \label{4:eq22.18}
\end{equation}%
Since there are limits \eqref{4:eq22.17} and
\begin{equation*}
\left\vert R_{2,z}e^{i\gamma _{n}^{(2)}}-R_{1,z}e^{i\gamma
_{n}^{(1)}}\right\vert
^{2}=R_{2,z}^{2}+R_{1,z}^{2}-2R_{1,z}R_{2,z}\cos (\gamma
_{n}^{(1)}-\gamma _{n}^{(2)})
\end{equation*}
and
\begin{equation*}
\cos (\gamma _{n}^{(1)}-\gamma _{n}^{(2)})=\cos \gamma
_{n}^{(1)}\cos \gamma _{n}^{(2)}+\sin \gamma _{n}^{(1)}\sin \gamma
_{n}^{(2)},
\end{equation*}%
limit \eqref{4:eq22.13} exists if and only if there is the limit
\begin{equation}
\underset{n\rightarrow \infty }{\lim }\sin \gamma _{n}^{(1)}\sin
\gamma _{n}^{(2)}.  \label{4:eq22.19}
\end{equation}%
Using \eqref{4:eq22.18} and \eqref{4:eq22.16} we obtain%
\begin{equation*}
\underset{n\rightarrow \infty }{\lim }\sin \gamma _{n}^{(1)}\sin
\gamma _{n}^{(2)}=\frac{\underset{n\rightarrow \infty }{\lim
}(\sin \gamma
_{n}^{(1)}\sin \theta _{n})(\sin \gamma _{n}^{(2)}\sin \theta _{n})}{%
\underset{n\rightarrow \infty }{\lim }(1-\cos ^{2}\theta _{n})}
\end{equation*}%
and
\begin{equation*}
\underset{n\rightarrow \infty }{\lim }\sin \gamma _{n}^{(j)}\sin \theta _{n}=%
\underset{n\rightarrow \infty }{\lim }\cos (\gamma _{n}^{(j)}-\theta _{n})-%
\underset{n\rightarrow \infty }{\lim }\cos\theta _{n}\cos \gamma
_{n}^{(j)}
\end{equation*}%
for $j=1,2.$ It implies the existence of \eqref{4:eq22.19}.
\end{proof}
\begin{remark}
\label{4:r2.8} Menger's notion of betweenness is well known for
the metric spaces, see, for example, \cite[p.~55]{Pa}. For points
$x,y,z$ belonging to a pseudometric space $(Y,d)$, we may say that
$x$ lies between $y$ and $z$ if
\begin{equation*}
 d(x,z)\cdot d(y,z)\neq 0\text{ and
}d(y,z)=d(y,x)+d(x,z).
\end{equation*}%
Suppose in Lemma \ref{4:l2.7} we have $\tilde d(\tilde 0, \tilde
x)\cdot \tilde d(\tilde 0,\tilde y)\neq 0,$ then inequality
\eqref{4:eq22.8} does not hold if and only if some point from the
set $\left\{\tilde{0},\tilde{x},\tilde{y} \right\}$ lies between
two other points of this set (in the pseudometric space
$(\tilde{X}_{0,\tilde{r}},\tilde{d})$).
\end{remark}

\begin{remark}
\label{4:r2.9} For the future it is useful to note that if $\beta
_{n}=0$ and $\theta _{n}\in \left[ 0,\pi \right] $ for all
$n\in\mathbb{N}$, then the sequences%
\begin{equation*}
\Big\{ \frac{|y_{n}|}{ r_{n} }e^{i\theta _{n}}\Big\}
_{n\in\mathbb{N}},\quad \bigg\{ \frac{\big\vert
z_{n}^{(1)}\big\vert }{r_{n}}e^{i\gamma _{n}^{(1)}}\bigg\} _{n\in
\mathbb{N}}\quad\text{ and }\quad\bigg\{ \frac{\big\vert z_{n}^{(2)}\big\vert }{r_{n}}%
e^{i\gamma _{n}^{(2)}}\bigg\} _{n\in\mathbb{N}},
\end{equation*}%
see, \eqref{4:eq22.11}, are convergent. Indeed, the function%
\begin{equation*}
\left[ 0,\pi \right] \ni t\longmapsto \cos t\in \left[ -1,1\right]
\end{equation*}%
is a homeomorphism. Hence $\left\{ \theta _{n}\right\}
_{n\in\mathbb{N}}$ is convergent because there is
$\underset{n\rightarrow \infty }{\lim} \cos \theta _{n}$. It
implies the convergence of  $\{\frac{|y_{n}|}{r_{n} }e^{i\theta
_{n}}\}_{n\in\mathbb N}$. Moreover, it
follows from \eqref{4:eq22.16} that%
\begin{equation*}
\underset{n\rightarrow \infty }{\lim }\sin (\theta _{n})\neq 0.
\end{equation*}%
This relation and the existence of limits \eqref{4:eq22.18},
\eqref{4:eq22.17} imply the convergence of the sequences \ $\{
\sin \gamma _{n}^{(1)}\}_{n\in\mathbb{N}}$ and $\{ \sin \gamma
_{n}^{(2)}\} _{n\in\mathbb{N}}$. Consequently $\{ \frac{\vert
z_{n}^{(1)}\vert}{r_{n}}e^{i\gamma _{n}^{(1)}}\}
_{n\in\mathbb{N}}$ and $\{\frac{\vert z_{n}^{(2)}\vert
}{r_{n}}e^{i\gamma _{n}^{(2)}}\} _{n\in \mathbb{N}}$ are also
convergent.
\end{remark}

\begin{lemma}
\label{4:l2.9} Let $X=\mathbb{C}$ be the set of all complex
numbers with the usual metric $\left\vert .,.\right\vert $ and
with the marked point $0$ and let
$\tilde{r}=\left\{r_{n}\right\}_{n\in \mathbb{N}} $ be a
normalizing sequence.  Let $\tilde{X}_{0, \tilde{r}}$ be a maximal
self-stable family for which \eqref{4:eq22.8} holds with some
$\tilde{x},\tilde{y}\in \tilde{X}_{0,\tilde{r}}$. Then there is a
maximal self-stable family
$\tilde{X}_{0,\tilde{r}}^{\ast}\subseteq\tilde X$ such that:

(i) $\tilde{X}_{0,\tilde{r}}^{\ast}$ and $\tilde{X}_{0,
\tilde{r}}$ are isometric;

(ii) There are $\tilde{x}^{\ast
}=\left\{x_{n}^{\ast}\right\}_{n\in\mathbb{N}}$ and
$\tilde{y}^{*}=\left\{ y_{n}^{\ast }\right\}_{n\in \mathbb{N}}$ in
$ \tilde{X}_{0,\tilde{r}}^{\ast}$ for which the inequality
\begin{equation}
2\max \left\{\tilde{d}(\tilde{0},\tilde{x}^{\ast}),\tilde{d
}(\tilde{0},\tilde{y}^{\ast }),\tilde{d}(\tilde{x}^{\ast },
\tilde{y}^{\ast })\right\} <\tilde{d}(\tilde{0},\tilde{x}%
^{\ast })+\tilde{d}(\tilde{0},\tilde{y}^{*})+\tilde{d}(
\tilde{x}^{*},\tilde{y}^{*})  \label{4:eq22.21}
\end{equation}%
holds and
\begin{equation*}
x_{n}^{\ast}=\left\vert x_{n}^{\ast}\right\vert ,~y_{n}^{\ast
}=\left\vert y_{n}^{\ast}\right\vert e^{i\theta_n} \text{ and
}\theta _{n}\in \left[ 0,\pi \right]
\end{equation*}%
for all $n\in\mathbb{N}$.
\end{lemma}

\begin{proof}
Let $\tilde{x}=\left\{ \left\vert x_{n}\right\vert e^{i\beta
_{n}}\right\}_{n\in \mathbb{N}}$ and $\tilde{y}=\left\{ \left\vert
y_{n}\right\vert e^{i\theta _{n}}\right\}_{n\in \mathbb{N}}$ be
elements of $\tilde{X}_{0,\tilde{r}}$ for which \eqref{4:eq22.8}
holds. Consider the sequence $\tilde{f}=\left\{ f_{n}\right\}
_{n\in\mathbb{N}}$ of the isometries

\begin{equation*}
f_n:\mathbb{C}\to\mathbb{C},\qquad f_{n}(z)=e^{-\beta_{n}}z.
\end{equation*}%
Then we have
\begin{equation*}
\tilde{f}(\tilde{x})=\left\{ \left\vert x_{n}\right\vert
\right\}_{n\in \mathbb{N}} \text{ and
}\tilde{f}(\tilde{y})=\left\{ \left\vert y_{n}\right\vert
e^{i(\theta _{n}-\beta _{n})}\right\} _{n\in\mathbb{N}}.
\end{equation*}%
We may assume that
\begin{equation*}
-\pi <\theta _{n}-\beta _{n}\leq \pi
\end{equation*}%
for all $n\in\mathbb{N}$. Define a new sequence $\tilde{g}$ of
isometries $g_{n}$ by the rule
\begin{equation*}
g_{n}(z):=\left\{
\begin{array}{ccc}
z & \text{if} & 0\leq \theta _{n}-\beta _{n}\leq \pi , \\
\overline{z} & \text{if} & -\pi <\theta _{n}-\beta _{n}<0.%
\end{array}%
\right.
\end{equation*}%
Using Lemma \ref{4:l2.8} we see that the family
\begin{equation*}
\tilde{X}_{0,\tilde{r}}^{\ast }:=\tilde{g}(\tilde{f}(
\tilde{X}_{0,\tilde{r}}))
\end{equation*}%
satisfies all desirable conditions with
\begin{equation*}
\tilde{x}^{\ast}:=\left\{ g_{n}(f_{n}(x_{n}))\right\}
_{n\in\mathbb{N}}\text{ and }\tilde{y}^{\ast }:=\left\{
g_{n}(f_{n}(y_{n}))\right\} _{n\in\mathbb{N}}.
\end{equation*}
\end{proof}

\begin{lemma}
\label{4:l2.10} Let $X=\mathbb{C}$ be the set of all complex
numbers with the usual metric $\left\vert .,.\right\vert $ and
$\tilde{r}=\left\{ r_{n}\right\}_{n\in\mathbb{N}}$ a normalizing
sequence and $\tilde{X}_{0,\tilde{r}}$ a maximal self-stable
family for which the conditions of Lemma \ref{4:l2.7} are
satisfied. If $\Omega _{0,\widetilde{r}}$ is a pretangent spaces
which corresponds $\tilde{X}_{0,\widetilde{r}},$ then
$\Omega_{0,\tilde{r}}$ is tangent and isometric to $\mathbb{C}$.
\end{lemma}

\begin{proof}
Let $\tilde{x}$ and $\tilde{y}$ be two elements of $\tilde{X}_{0,
\tilde{r}}$ for which \eqref{4:eq22.8} holds. By
Lemma~\ref{4:l2.9} there exists a maximal self-stable family
$\tilde{X}_{0,\tilde{r} }^{\ast }\subseteq\tilde X$ which is
isometric to $\tilde{X}_{0,\tilde{r}}$ and contains
$\tilde{x}^{\ast }=\left\{ x_{n}^{\ast }\right\}
_{n\in\mathbb{N}}$ and
$\tilde{y}^{\ast}=\left\{y_{n}^{\ast}\right\}_{n\in\mathbb{N}}$
such that \eqref{4:eq22.21} holds and
\begin{equation}
x_{n}^{\ast }=\left\vert x_{n}^{\ast }\right\vert ,~~y_{n}^{\ast
}=\left\vert y_{n}^{\ast }\right\vert e^{i\theta _{n}},~\theta
_{n}\in \left[ 0,\pi \right]   \label{4:eq22.22}
\end{equation}%
for all $n\in\mathbb{N}.$ Relations \eqref{4:eq22.22} imply that
for every \ $\tilde{z}^{\ast }=\left\{ z_{n}^{\ast }\right\}
_{n\in \mathbb{N}
}\in \tilde{X}_{0,\tilde{r}}^{*}$ the sequence%
\begin{equation*}
\frac{\tilde z^{*}}{\tilde r}:=\left\{\frac{z_{n}^{\ast
}}{r_{n}}\right\}_{n\in \mathbb{N} }
\end{equation*}%
is convergent, see Remark \ref{4:r2.9} . Write%
\begin{equation*}
z^{\ast}:=\underset{n\rightarrow \infty }{\lim }\frac{z_{n}^{\ast
}}{r_{n}}
\end{equation*}%
for every $\tilde{z}^{\ast }=\left\{ z_{n}^{\ast }\right\}
_{n\in\mathbb{N}}\in \tilde{X}_{0,\tilde{r}}^{*}$. In particular,
we have
\begin{equation}
x^{\ast}:=\underset{n\rightarrow \infty }{\lim }\frac{x_{n}^{\ast
}}{r_{n}}\,\,\, \text{and}\,\,\, y^{\ast }:=\underset{n\rightarrow
\infty }{\lim }\frac{y_{n}^{\ast }}{r_{n}}.  \label{4:eq22.23}
\end{equation}%
We claim that the function%
\begin{equation}
\tilde{X}_{0,\tilde{r}}^{*}\ni \tilde{z}^{\ast }\overset{f}{%
\longmapsto}z^{\ast }\in \mathbb{C}
\label{4:eq22.24}
\end{equation}%
is distance-preserving and onto. (It immediately implies that
$\Omega _{0,\tilde{r}}^{\ast}$, the metric identification of
$\tilde{X}_{0, \tilde{r}}^{*}$, and $\mathbb{C}$ are isometric, so
$\Omega _{0,\tilde{r}}$ also is isometric to $\mathbb{C}.$)
Indeed, if
$\tilde{w}^{\ast}=\left\{w_{n}\right\}_{n\in\mathbb{N}}\in
\tilde{X}_{0,\tilde{r}}^{\ast}$ ,  then
\begin{equation*}
\tilde{d}(\tilde{w}^{\ast },\tilde{z}^{\ast})=\underset{%
n\rightarrow \infty}{\lim}\frac{\left\vert w_{n}^{\ast
}-z_{n}^{\ast
}\right\vert }{r_{n}}=\underset{n\rightarrow \infty }{\lim }\left\vert \frac{%
w_{n}^{\ast }}{r_{n}}-\frac{z_{n}^{\ast }}{r_{n}}\right\vert
=\left\vert w^{\ast }-z^{\ast }\right\vert .
\end{equation*}%
Consequently it is sufficient to show that for every
$p\in\mathbb{C}$ there is $\tilde{p}^{\ast}\in
\tilde{X}_{0,\tilde{r}}^{\ast}$ such that $p^{\ast}=p.$ Write
\begin{equation*}
\tilde{p}^{\ast}=\left\{r_{n}p\right\}_{n\in\mathbb{N}}.
\end{equation*}%
It is clear that
\begin{equation*}
p^{\ast}=\underset{n\rightarrow \infty }{\lim }\frac{r_{n}p}{%
r_{n}}=p,
\end{equation*}%
thus it is enough to  prove that $\tilde{p}^{\ast }\in
\tilde{X}_{0,\tilde{r}}^{\ast}$.  It follows from \eqref{4:eq22.23} that%
\begin{equation*}
\underset{n\rightarrow \infty}{\lim}\frac{\left\vert r_{n}x^{\ast
}-x_{n}^{\ast }\right\vert }{r_{n}}=\underset{n\rightarrow \infty}{\lim}%
\frac{\left\vert r_{n}y^{\ast }-y_{n}^{\ast }\right\vert
}{r_{n}}=0.
\end{equation*}%
Hence%
\begin{equation*}
\tilde{d}(\tilde{p}^{\ast },\tilde{x}^{\ast })=\underset{%
n\rightarrow \infty }{\lim }\frac{\left\vert r_{n}p-x_{n}^{\ast
}\right\vert }{r_{n}}=\underset{n\rightarrow \infty }{\lim
}\frac{\left\vert r_{n}p-r_{n}x^{\ast }\right\vert
}{r_{n}}=\left\vert p-x^{\ast }\right\vert
\end{equation*}%
and similarly we have $\tilde{d}(\tilde{p}^{\ast },\tilde{y}%
^{\ast })=\left\vert p-\tilde{y}^{\ast }\right\vert .$ Therefore,
by Lemma~\ref{4:l2.7}, $\tilde{p}^{\ast }\in
\tilde{X}_{0,\tilde{r}}^{\ast }.$

It remains to show that $\Omega_{0,\tilde{r}}^{\ast }$ is tangent,
because, by Lemma \ref{4:l2.8}, in this case $\Omega_{0,\tilde r}$
is also tangent. Let $\tilde{r}^{\prime }=\left\{
r_{n_{k}}\right\}_{k\in\mathbb{N}}$ be a subsequence of
$\tilde{r}$ and let $\tilde{X}_{0,\tilde{r }^{\prime}}^{\ast}$ be
maximal self-stable family such that
\begin{equation*}
\tilde{X}_{0,\tilde{r}^{\prime }}^{\ast}\supseteq \left\{ \left\{
x_{n_{k}}^{\ast }\right\}_{k\in\mathbb{N}}:\left\{ x_{n}^{\ast
}\right\}_{n\in\mathbb{N}}\in \tilde{X}_{0,\tilde{r}}^{\ast
}\right\}.
\end{equation*}%
Then $\tilde{X}_{0,\tilde{r}^{\prime}}^{\ast}$ satisfies all
conditions of the lemma  which is being proved. Hence,
similarly \eqref{4:eq22.24}, we can define a function%
\begin{equation*}
\tilde{X}_{0,\tilde{r}^{\prime }}^{\ast }\ni \tilde{z}^{\ast }%
\overset{f^{\prime }}{\longmapsto}z^{\ast }\in \mathbb{C} ,\qquad
f^{\prime }(\tilde{z}^{\ast })=z^{\ast }=\underset{n\rightarrow
\infty}{\lim }\frac{x_{k}^{\ast }}{r_{n_{k}}}
\end{equation*}%
for $\tilde{z}^{\ast }=\left\{z_{k}\right\}_{k\in \mathbb{N}
}\in X_{0,\tilde{r}^{\prime }}^{\ast }.$ Let $\Omega _{0,\tilde{r}%
^{\prime }}^{\ast }$ be the metric identification of $\tilde{X}_{0,%
\tilde{r}^{\prime }}^{\ast }$ and let $is:\mathbb{C}
\longrightarrow \Omega^*_{0,\tilde{r}},$
$is':\mathbb{C}\longrightarrow \Omega _{0,\tilde{r'}}^{\ast }$ be
isometries such that the diagrams
\begin{equation}\label{4:eq2.25}
\begin{diagram}
\node{\tilde X^*_{0,\tilde r}}\arrow[2]{e,t}{f}
                              \arrow{se,t}{p}\node[2]{\mathbb
                              C}\arrow{sw,r}{is}
                              \\\node[2]{\Omega^*_{0,\tilde
                              r}}
\end{diagram}
\qquad\text{and}\qquad
\begin{diagram}
\node{\tilde X^*_{0,\tilde r'}}\arrow[2]{e,t}{f'}
                              \arrow{se,t}{p'}\node[2]{\mathbb
                              C}\arrow{sw,r}{is'}
                              \\\node[2]{\Omega^*_{0,\tilde
                              r'}}
\end{diagram}
\end{equation}
are commutative. Similarly \eqref{1:eq1.4} we can define an
isometric embedding $\text{em}':\Omega^*_{0,\tilde
r}\to\Omega^*_{0,\tilde r'}$ such that the diagram
\begin{equation}\label{4:eq2.26*}
\begin{diagram}
\node{\tilde X^*_{0,\tilde r}}\arrow[2]{e,t}{\text{in}_{\tilde
r'}}
                       \arrow[2]{s,l}{p}\node[2]{\tilde X^*_{0,\tilde
                       r'}}
                       \arrow[2]{s,r}{p'}\\ \\
\node{\Omega^*_{0,\tilde
r}}\arrow[2]{e,b}{em'}\node[2]{\Omega^*_{0,\tilde r'}}
\end{diagram}
\end{equation}
is commutative. We claim that the following diagram
\begin{equation}\label{4:eq2.26}
\begin{diagram}
\node{\tilde X^*_{0,\tilde r}}\arrow[2]{e,t}{\text{in}_{\tilde
r'}}
                       \arrow{se,t}{f}
                       \arrow[2]{s,l}{p}\node[2]{\tilde X^*_{0,\tilde
                       r'}} \arrow{sw,l}{f'}
                       \arrow[2]{s,r}{p'}\\
\node[2]{\mathbb C}     \arrow{sw,r}{is}
                     \arrow{se,b}{is'}\\
\node{\Omega^*_{0,\tilde
r}}\arrow[2]{e,b}{em'}\node[2]{\Omega^*_{0,\tilde r'}}
\end{diagram}
\end{equation}
also is commutative. To prove the commutativity of
\eqref{4:eq2.26} it is sufficient to show that
\begin{equation}
em^{\prime }(is(z))=is^{\prime }(z)  \label{4:eq22.27*}
\end{equation}%
for every $z\in \mathbb{C}
.$ Let $z$ be a point of $%
\mathbb{C}
.$ Since $f$ is a surjection, there is $\tilde{x}\in \tilde{X}_{0,%
\tilde{r}}^{\ast }$ such that $ z=f(\tilde{x}). $ Hence, using the
commutativity of diagrams \eqref{4:eq2.25}--\eqref{4:eq2.26*}  and
and the equality $f=f^{\prime }\circ
in_{\tilde r^{\prime }}$ we obtain%
\begin{equation*}
em^{\prime}(is(z))=em^{\prime}(is(f(\tilde{x})))=em^{\prime}(p(%
\tilde{x}))=
\end{equation*}%
\begin{equation*}
p^{\prime}(in_{\tilde
r^{\prime}}(\tilde{x}))=is^{\prime}(f^{\prime }(in_{\tilde
r^{\prime }}(\tilde{x}))=is^{\prime}(f(\tilde{x}))=is^{\prime
}(z).
\end{equation*}%
Consequently \eqref{4:eq22.27*} holds. The diagram%
\begin{equation*}\label{4:tri_diag}
\begin{diagram}
\node{\tilde X^*_{0,\tilde r}}   \arrow[2]{s,l}{p}
                          \arrow{ese,t}{f}\\
\node[3]{\mathbb C}       \arrow{ese,t}{is'}\\
\node{\Omega^*_{a,\tilde r}}
                            \arrow[3]{e,t}{em'}
\node[4]{\Omega^*_{a,\tilde r'}}
\end{diagram}
\end{equation*}
also is commutative, because \eqref{4:eq2.26} is commutative. Since $f$ and $%
is^{\prime }$ are surjections, $\ em^{\prime }$ is surjective.
Hence, by Proposition \ref{1:p1.5},  $\Omega_{0,\tilde{r}}^{\ast
}$ is tangent.
\end{proof}

The following lemma shows that if $X=\mathbb{C}$, then every
maximal self-stable $\tilde{X}_{0,\tilde{r}}$ satisfies the
conditions of Lemma \ref{4:l2.7}. It is a final part of the proof
of Proposition~\ref{4:p2.6}.

\begin{lemma}
\label{4:l2.13} Let $X=\mathbb{C}$ be the set of all complex
numbers with the usual metric $\left\vert .,.\right\vert ,$ let be
$\tilde{r}=\left\{ r_{n}\right\} _{n\in \mathbb{N} }$ a
normalizing sequence and let $\tilde{X}_{0,\tilde{r}}$ be a
maximal self-stable family. Then there are $\tilde{x},\tilde{
y}\in \tilde{X}_{0,\tilde{r}}$ such that \eqref{4:eq22.8} holds, i.e.,%
\begin{equation*}
2\max \left\{ \tilde{d}(\tilde{0},\tilde{x}),\tilde{d}(%
\tilde{0},\tilde{y}),\tilde{d}(\tilde{x},\tilde{y}%
)\right\} <\tilde{d}(\tilde{0},\tilde{x})+\tilde{d}(%
\tilde{0},\tilde{y})+\tilde{d}(\tilde{x},\tilde{y}).
\end{equation*}
\end{lemma}

\begin{proof}
Suppose that the equality
\begin{equation}
2\max \left\{ \tilde{d}(\tilde{0},\tilde{x}),\tilde{d}(%
\tilde{0},\tilde{y}),\tilde{d}(\tilde{x},\tilde{y}%
)\right\} =\tilde{d}(\tilde{0},\tilde{x})+\tilde{d}(%
\tilde{0},\tilde{y})+\tilde{d}(\tilde{x},\tilde{y})
\label{4:eq22.27}
\end{equation}%
holds for all $\tilde{x},\tilde{y}\in \tilde{X}_{0,\tilde{r}}$.

Consider first the simplest case where $\tilde{d}(\tilde{0},%
\tilde{x})=0$ for all $\tilde{x}\in \tilde{X}_{0,\tilde{r}}.$ The
last equality implies
\begin{equation*}
\tilde{d}(\tilde{r},\tilde{x})=\underset{n\rightarrow \infty }{%
\lim }\frac{\tilde{d}(x_n,r_{n})}{r_{n}}=\underset{%
n\rightarrow \infty }{\lim }\left\vert
\frac{x_{n}}{r_{n}}-1\right\vert =1
\end{equation*}%
for all $\tilde{x}=\left\{ x_{n}\right\} _{n\in\mathbb{N}}\in
\tilde{X}_{0,\tilde{r}}$. Consequently $\tilde{X}_{0,
\tilde{r}}\cup \left\{\tilde{r}\right\} $ is a self-stable family,
so $\tilde{X}_{0,\tilde{r}}$ is not maximal self-stable, contrary
to the conditions.

Hence there is $\tilde{x}=\left\{x_{n}\right\}_{n\in\mathbb{N}}\in
\tilde{X}_{0,\tilde{r}}$ such that
\begin{equation}
\tilde{d}(\tilde{0},\tilde{x})=c_{0}>0.  \label{4:eq22.28}
\end{equation}
Without loos of generality we may suppose that
\begin{equation}
\tilde{x}=\tilde{r}.  \label{4:eq22.29}
\end{equation}%
Indeed, passing, if necessary, to an isometric
$\tilde{X}_{0,\tilde{r}}^{\ast}$ , see Lemma \ref{4:l2.8}, we may
put $x_{n}=\left\vert x_{n}\right\vert $ for all $n\in\mathbb{N}$.
Next, since the family
\begin{equation*}
c\tilde{X}_{0,\tilde{r}}=\left\{ \left\{ cz_{n}\right\} _{n\in
\mathbb{N}}:\left\{ z_{n}\right\} _{n\in \mathbb{N} }\in
\tilde{X}_{0,\tilde{r}}\right\}
\end{equation*}%
is maximal self-stable if $0\neq c\in\mathbb{C},$ we can take
$\frac{1}{c_{0}}\tilde{X}_{0,\tilde{r}}$ instead of $
\tilde{X}_{0,\tilde{r}}$. Moreover, since
\begin{equation*}
\frac{1}{c_{0}}\tilde{x}:=\left\{
\frac{1}{c_{0}}x_{n}\right\}_{n\in \mathbb{N}}\in
\frac{1}{c_{0}}\tilde{X}_{0,\tilde{r}}\,\,\,\text{and}\,\,\,
\underset{ n\rightarrow \infty }{\lim
}\frac{1}{c_{0}}\frac{x_{n}}{r_{n}}=1,
\end{equation*}%
we see that $\tilde{d}(\frac{1}{c_{0}}\tilde{x},\tilde{r})=0$ and,
consequently,
\begin{equation*}
\tilde{r}\in \frac{1}{c_{0}}\tilde{X}_{0,\tilde{r}}.
\end{equation*}%
The equalities \eqref{4:eq22.27} and \eqref{4:eq22.29} imply that
\begin{equation}
2\max \left\{1,\tilde{d}(\tilde{0},\tilde{y}),\tilde{d}(%
\tilde{r},\tilde{y})\right\} =1+\tilde{d}(\tilde{0},%
\tilde{y})+\tilde{d}(\tilde{r},\tilde{y}) \label{4:eq22.30}
\end{equation}%
for all $\tilde{y}\in \tilde{X}_{0,\tilde{r}}.$ There exist only
the  following three possibilities under which \eqref{4:eq22.30}
holds:
\begin{equation}
\tilde{d}(\tilde{r},\tilde{y})=\tilde{d}(\tilde{y},%
\tilde{0})+1,~~1=\tilde{d}(\tilde{0},\tilde{y})+\tilde{d}%
(\tilde{y},\tilde{r})\text{ and }\tilde{d}(\tilde{0},%
\tilde{y})=1+\tilde{d}(\tilde{r},\tilde{y}). \label{4:eq22.31}
\end{equation}%
Write for $\tilde y\in\tilde X_{0,\tilde r}$
\begin{equation}
t=t(\tilde y):=\left\{
\begin{array}{cc}
-\tilde{d}(\tilde{y},\tilde{0}) &\text{if \ } \tilde{d}(\tilde{r},%
\tilde{y})=\tilde{d}(\tilde{y},\tilde{0})+1, \\
\tilde{d}(\tilde{0},\tilde{y}) & \text{otherwise}.%
\end{array}%
\right.  \label{4:eq22.32}
\end{equation}%
We claim that the limit relation
\begin{equation}
\underset{n\rightarrow \infty }{\lim}\frac{\left\vert
tr_{n}-y_{n}\right\vert }{r_{n}}=0  \label{4:eq22.33}
\end{equation}%
holds for every $\tilde{y}=\left\{
y_{n}\right\}_{n\in\mathbb{N}}\in \tilde{X}_{0,\tilde{r}}$.
Indeed, fix $\tilde y\in\tilde X_{0,\tilde r}$, and suppose that
the first equality from \eqref{4:eq22.31} holds. (Note that
\eqref{4:eq22.33} is
trivial if $\tilde{d}(\tilde 0,\tilde{y})=0$ or $\tilde{d}(\tilde{r}%
,\tilde{y})=0$.) Let us denote by $y^{\ast }$ an arbitrary limit
point of the sequence $\left\{ \frac{y_{n}}{r_{n}}\right\} _{n\in
\mathbb{N}}.$ To prove \eqref{4:eq22.33} it is sufficient to show
that
\begin{equation}
y^{\ast }=t.  \label{4:eq22.34}
\end{equation}%

\begin{figure}[h]\centering
\includegraphics[width=9cm,keepaspectratio]{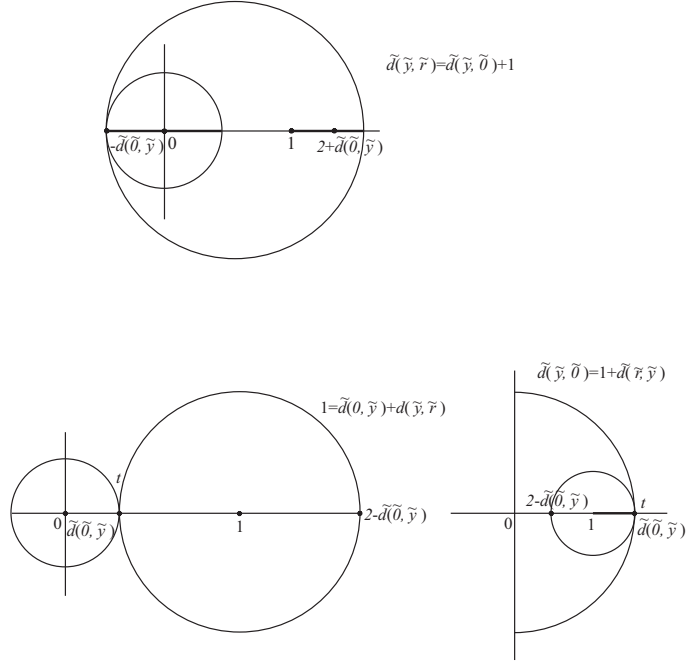}
\caption{Points $\tilde x,\tilde y$ and $\tilde 0$ are situated on
the ``real axis''.}
\end{figure}
The fist equality in \eqref{4:eq22.31} and definition
\eqref{4:eq22.32} imply that
\begin{equation*}
\left\vert y^{\ast }\right\vert =\tilde{d}(\tilde{y},\tilde{0}%
)=\left\vert t\right\vert \text{ and }\left\vert y^{\ast }-1\right\vert =%
\underset{n\rightarrow \infty }{\lim }\frac{\left\vert
y_{n}-r_{n}\right\vert }{r_{n}}=\tilde{d}(\tilde{r},\tilde{y})=1+%
\tilde{d}(\tilde{y},\tilde 0).
\end{equation*}%
Hence $y^{\ast}$ belongs to the intersection of the circumferences
\begin{equation*}
\left\{ z\in \mathbb{C} :\left\vert z\right\vert
=\tilde{d}(\tilde{0},\tilde{y})\right\} \text{ and }\left\{ z\in
\mathbb{C}:\left\vert z-1\right\vert =1+\tilde{d}(\tilde{y},\tilde{0}%
)\right\} .
\end{equation*}
Since this intersection has the unique element $t,$ see Fig.~1, we
obtain \eqref{4:eq22.34}. Similarly  we have
\begin{equation*}
y^{\ast}\in \left\{z\in\mathbb{C}:\left\vert z\right\vert
=\tilde{d}(\widetilde{0},\tilde{y})\right\} \cap
\left\{z\in\mathbb{C}:\left\vert z-1\right\vert
=1-\tilde{d}(\tilde{0},\tilde{y})\right\} =\left\{ t\right\}
\end{equation*}
if $1=\tilde{d}(\tilde{0},\tilde{y})+\tilde{
d}(\tilde{y},\tilde{r})$ and
\begin{equation*}
y^{\ast }\in \left\{ z\in\mathbb{C}:\left\vert z\right\vert
=\tilde{d}(\tilde{0},\tilde{y})\right\} \cap \left\{
z\in\mathbb{C}:\left\vert z-1\right\vert
=\tilde{d}(\tilde{0},\tilde{y})-1\right\}=\{t\}
\end{equation*}%
for the case   $\tilde{d}(\tilde{0},\tilde{%
y})=1+\tilde{d}(\tilde{y},\tilde{z})$, i.e., $y^{\ast }=t$ for all
possible cases.

To complete the proof let us consider the family
\begin{equation*}
\tilde{X}_{0,\tilde{r}}^{\prime }:=\left\{ \left\{ cr_{n}\right\}
_{n\in\mathbb{N}}:c\in\mathbb{C}\right\}.
\end{equation*}%
Since \eqref{4:eq22.33} holds for all $\tilde{y}\in
\tilde{X}_{0,\tilde{r}}$, it is easy to prove that
$[\tilde{X}_{0,\tilde r}^{\prime }]_X\supseteq \tilde{X}_{0,\tilde
r}$ where the operation $[\;\cdot\;]_X$ was defined in
\eqref{6:eq5.1}. Note that $[\tilde X_{0,\tilde r}']_X$ is
self-stable and that
\begin{equation*}
\tilde{d}(i\tilde{r},t\tilde{r})= \underset{n\rightarrow \infty }{%
\lim }\frac{\left\vert ir_{n}-tr_{n}\right\vert
}{r_{n}}=\sqrt{1+t^2} \neq 0
\end{equation*}
for every $t\in\mathbb{R}$. Hence $i\tilde{r}\notin
\tilde{X}_{0,\tilde{r}},$ contrary to the supposition about the
maximality of $\tilde{X}_{0,\tilde{r}}.$
\end{proof}

\section{Tangent  spaces to starlike sets}

Examples \ref{4:e2.1} and \ref{4:e2.5} are some particular cases
of starlike sets on the plane. The next our goal is to prove
Theorem \ref{t:1.7} which describes tangent spaces to  arbitrary
starlike subsets of $\mathbb C$.

For convenience we repeat this theorem here.
\begin{theorem}\label{t:4.1}
Let $Y\subseteq\mathbb C$ be a starlike set with a center $a$ and
let $\tilde r$ be a normalizing sequence. Then all pretangent
spaces to $Y$ at the point $a$ lie in tangent spaces and for each
tangent space $\Omega_{a,\tilde r}^Y$ there is an isometry
\begin{equation}\label{eq4.1}
\psi:\Omega_{a,\tilde r}^Y\to Con_a(Y)\quad\text{with
}\psi(\alpha)= a.
\end{equation}
\end{theorem}

Before proving the theorem we consider its particular cases. If
$Y$ is an one-point set, then all rays \eqref{8:eq22.34} are also
one-point and so $Con_{a}(Y)=\left\{ a\right\} $. Moreover, it is
easy to see that, in this case, all pretangent spaces are tangent
and one-point. Hence the theorem is valid if $Y=\left\{
a\right\}$. For the case when each three point of $Y$ are
collinear, desirable isometry \eqref{eq4.1} was, in fact,
constructed in the proof of Proposition \ref{4:p2.4}, see diagram
\eqref{4:eq2.7}. If $Int(Y)\neq \emptyset $ and $a\in Int(Y),$
then the theorem follows from Proposition \ref{4:p2.6}.
Consequently it is enough examine the case where $a\in
\partial Y$ and $Y$ contains at least three noncollinear points.

\begin{lemma}\label{8:l2.16}
Let $X\subseteq \mathbb C$ be a closed cone with a vertex $0$. Let
$a=0$ be a marked point of $X,$ $\tilde{r}=\left\{ r_{n}\right\}
_{n\in \mathbb{N} }$ a normalizing sequence and
$\tilde{X}_{0,\tilde{r}}$ a maximal self-stable family. Suppose
$X$ contains at least three noncollinear points and  there exist
$\tilde{x}=\left\{ x_{n}\right\} _{n\in \mathbb{N} }$ and \
$\tilde{y}=\left\{ y_{n}\right\} _{n\in \mathbb{N} }$ from
$\tilde{X}_{0,\tilde{r}}$ such that
\begin{equation}\label{8:eq8.3}
2\max\{\tilde d(\tilde0,\tilde x),\tilde d (\tilde0, \tilde y),
\tilde d(\tilde x, \tilde y)\}<\tilde d(\tilde0,\tilde x)+\tilde
d(\tilde0,\tilde y)+\tilde d(\tilde x, \tilde y)
\end{equation}
 holds
and sequences
\begin{equation}
\frac{\tilde{x}}{\tilde{r}}:=\left\{\frac{x_{n}}{r_{n}}\right\}
_{n\in \mathbb{N}
},~~\frac{\tilde{y}}{\tilde{r}}:=\left\{\frac{y_{n}}{r_{n}}\right\}
_{n\in \mathbb{N} }  \label{8:eq22.36}
\end{equation}%
are convergent. Then the following statements hold.

(i) If $\tilde{z}=\left\{ z_{n}\right\} _{n\in \mathbb{N} }\in
\tilde{X}_{0,\tilde{r}},$ then there exists a limit
\begin{equation}
z^{\ast }=\underset{n\rightarrow \infty }{\lim
}\frac{z_{n}}{r_{n}} \label{8:eq22.37}
\end{equation}%
and%
\begin{equation}
z^{\ast }\in X.  \label{8:eq22.38}
\end{equation}

(ii) Conversely, if $\tilde{z}\in \tilde{X}$ and if relations
\eqref{8:eq22.37}, \eqref{8:eq22.38} hold, then $\tilde{z}\in \tilde{X}_{0,%
\tilde{r}}.$

(iii) The mapping $\tilde{X}_{0,\tilde{r}}\overset{f}{%
\longrightarrow }X,~~f(\tilde{z})=z^{\ast },$ is
distance-preserving and onto.
\end{lemma}

\begin{proof}
Suppose that
\begin{equation*}
\tilde{z}=\left\{\left\vert z_{n}\right\vert e^{i\gamma
_{n}}\right\}_{n\in \mathbb{N} }\in \tilde{X}_{0,\tilde{r}}.
\end{equation*}%
Since $\tilde{x}$ and $\tilde{z}$ are mutually stable, there is a
limit
\begin{equation}
\underset{n\rightarrow \infty }{\lim }\left\vert \frac{z_{n}}{r_{n}}-\frac{%
x_{n}}{r_{n}}\right\vert ^{2}
=R_{x}^{2}+R_{z}^{2}-2R_{x}R_{z}\underset{n\rightarrow \infty
}{\lim }\cos (\gamma _{n}-\beta )  \label{8:eq22.39}
\end{equation}%
where%
\begin{equation*}
R_{z}=\underset{n\rightarrow \infty }{\lim }\frac{\left\vert
z_{n}\right\vert }{r_{n}},~R_{x}=\underset{n\rightarrow \infty }{\lim }\frac{%
\left\vert x_{n}\right\vert }{r_{n}},~R_{x}e^{i\beta }=\underset{%
n\rightarrow \infty }{\lim }\frac{x_{n}}{r_{n}}.
\end{equation*}%
Similarly we have
\begin{equation}
\underset{n\rightarrow \infty }{\lim }\left\vert \frac{z_{n}}{r_{n}}-\frac{%
y_{n}}{r_{n}}\right\vert ^{2}=R_{y}^{2}+R_{z}^{2}-2R_{y}R_{z}\underset{%
n\rightarrow \infty }{\lim }\cos (\gamma _{n}-\theta )
\label{8:eq22.40}
\end{equation}%
where%
\begin{equation*}
R_{y}=\underset{n\rightarrow \infty }{\lim }\frac{\left\vert
y_{n}\right\vert }{r_{n}}\text{ and }~R_{y}e^{i\theta }=\underset{%
n\rightarrow \infty }{\lim }\frac{y_{n}}{r_{n}}.
\end{equation*}%
Consider the system
\begin{eqnarray}
\cos \theta \cos \gamma _{n}+\sin \theta \sin \gamma _{n} &=&\cos
(\gamma
_{n}-\theta )  \label{8:eq22.41} \\
\cos \beta \cos \gamma _{n}+\sin \beta \sin \gamma _{n} &=&\cos
(\gamma _{n}-\beta ).  \notag
\end{eqnarray}%
The inequality \eqref{8:eq8.3} implies that%
\begin{equation*}
\left\vert
\begin{array}{cc}
\cos \theta & \sin \theta \\
\cos \beta & \sin \beta%
\end{array}%
\right\vert =\sin (\beta -\theta )\neq 0.
\end{equation*}%
Hence, by Cramer's rule, we obtain from \eqref{8:eq22.41}%
\begin{equation*}
\cos \gamma _{n}=\frac{\cos (\gamma _{n}-\theta )\sin \beta -\cos
(\gamma _{n}-\beta )\sin \theta }{\sin (\beta -\theta )},
\end{equation*}%
\begin{equation*}
\sin \gamma _{n}=\frac{\cos (\gamma _{n}-\beta )\cos \theta -\cos
(\gamma _{n}-\theta )\cos \beta }{\sin (\beta -\theta )}.
\end{equation*}%
Consequently, the existence of limits \eqref{8:eq22.39} and
\eqref{8:eq22.40} implies the existence of limit
\eqref{8:eq22.37}. Note also that the elements of the sequence
$\frac{\tilde{z}}{\tilde{r}}$ are some points of $X$. Hence we
have \eqref{8:eq22.38} because $X$ contains all its limit points
as a closed set.

Now suppose that $\tilde{z}\in \tilde{X}$ and relations
\eqref{8:eq22.37}, \eqref{8:eq22.38} hold. Since
$\tilde{X}_{0,\tilde{r}}$
is maximal self-stable, to prove $\tilde{z}\in \tilde{X}_{0,%
\tilde{r}}$ it is sufficient to show that that $\tilde{z}$ and $%
\tilde{w}=\left\{ w_{n}\right\}_{n\in\mathbb{N}}$ are mutually
stable for each $\tilde{w}\in \tilde{X}_{0,
\tilde{r}}.$ Let $\tilde{w}$ be an arbitrary element of $\tilde{X%
}_{0,\tilde{r}}$.  Statement (i) implies that there is
$w^{\ast}\in X$ such that
\begin{equation*}
\underset{n\rightarrow \infty }{\lim }\frac{\left\vert
w_{n}-r_{n}w^{\ast }\right\vert }{r_{n}}=0.
\end{equation*}
Hence, by \eqref{8:eq22.37},
\begin{equation}
\tilde{d}(\tilde{w},\tilde{z})=\underset{n\rightarrow \infty }{%
\lim }\frac{\left\vert z_{n}-w_{n}\right\vert }{r_{n}}=\left\vert
z^{\ast }-w^{\ast }\right\vert , \label{8:eq22.42}
\end{equation}%
i.e., $\tilde{z}$ and $\tilde{w}$ are mutually stable.

To prove Statement (iii) note that \eqref{8:eq22.42} means that
the function
\begin{equation*}
\tilde{X}_{0,\tilde{r}}\ni \tilde{z}\overset{f}{\longmapsto}%
z^{\ast }\in X
\end{equation*}%
is distance-preserving. Moreover, Statement (ii) implies that  for
every $z^{\ast }\in X$ we have $\left\{ z^{\ast }r_{n}\right\}
_{n\in\mathbb{N} }\in \tilde{X}_{0,\tilde{r}},$  i.e., $f$ is
onto.
\end{proof}

A modification of the proof of Lemma \ref{4:l2.13} gives the
following.

\begin{lemma}
\label{8:l2.17} Let $X$ be a set from Lemma \ref{8:l2.16}. Then
for every
maximal self-stable $\tilde{X}_{0,\tilde{r}}$ there are $\tilde{r%
}^{\prime }$ and $\tilde{X}_{0,\tilde{r}^{\prime}}$ such that
\begin{equation*}
\tilde{X}_{0,\tilde{r}^{\prime }}\supseteq \left\{\tilde{x}%
^{\prime }:\tilde{x}\in \tilde{X}_{0,\tilde{r}}\right\}
\end{equation*}%
and \eqref{8:eq8.3} holds \ for some $\tilde{x},\tilde{y}\in
\tilde{X}_{0,\tilde{r}^{\prime }}.$
\end{lemma}

The following proposition is a model case of Theorem \ref{t:4.1}.

\begin{proposition}
\label{8:p2.18} Let $X$ be a set from Lemma \ref{8:l2.16}. Then
the conclusion of Theorem \ref{t:4.1} is valid for every
$\Omega_{0, \tilde{r}}.$
\end{proposition}

\begin{proof}
Let $\tilde r$ be a normalizing sequence and let $\tilde
X_{0,\tilde r} $ be a maximal self-stable family with
 a corresponding pretangent space $\Omega_{0,\tilde r} $. By Lemma
\ref{8:l2.17}  we may suppose, passing, if necessary, to a
subsequence of $\tilde r$, that there exist $\tilde{x}=\left\{
x_{n}\right\} _{n\in\mathbb{N}}$ and \ $\tilde{y}=\left\{
y_{n}\right\} _{n\in\mathbb{N}}$ in $\tilde{X}_{0,\tilde{r}}$ such
that \eqref{8:eq8.3} holds.
Since the sequences%
\begin{equation*}
\left\{ \frac{x_{n}}{r_{n}}\right\}_{n\in\mathbb{N}},~~\left\{
\frac{y_{n}}{r_{n}}\right\} _{n\in\mathbb{N}}
\end{equation*}%
are bounded, there is a subsequences $\tilde{r}^{\prime }=\left\{
r_{n_{k}}\right\} _{k\in\mathbb{N}}$ of sequence $\tilde{r}$  such
that
\begin{equation*}
\left\{ \frac{x_{n_{k}}}{r_{n_{k}}}\right\} _{k\in\mathbb{N}
},~~\left\{ \frac{y_{n_{k}}}{r_{n_{k}}}\right\} _{k\in\mathbb{N}}
\end{equation*}%
are convergent. Let  $\tilde{X}_{0,\tilde{r}^{\prime }}$ be a
maximal self-stable  family such that
\begin{equation}\label{*5}
\tilde{X}_{0,\tilde{r}^{\prime }}\supseteq
\left\{\tilde{x}^{\prime }:\tilde{x}\in
\tilde{X}_{0,\tilde{r}}\right\}.
\end{equation}%
Write  $\Omega _{0,\tilde{r}^{\prime}}$ for the metric
identifications of  of $\tilde{X}_{0,\tilde{r}^{\prime}}$. We
claim that $\Omega _{0,\tilde{r}^{\prime }}$ is tangent.
Indeed, since all suppositions of Lemma \ref{8:l2.16} are valid, a limit%
\begin{equation*}
z^{\ast }=\underset{k\rightarrow \infty }{\lim
}\frac{z_{k}}{r_{n_{k}}}
\end{equation*}%
exists for every $\tilde{z}=\left\{ z_{k}\right\} _{k\in\mathbb{N}
}\in \tilde{X}_{0,\tilde{r}^{\prime}}$ and the mapping
\begin{equation}\label{*6}
\tilde{X}_{0,\tilde{r}^{\prime }}\ni \tilde{z}\overset{f^{\prime
}}{\longmapsto }z^{\ast}\in X
\end{equation}%
is distance-preserving and onto. Hence there is an isometry
$is^{\prime}:X\longrightarrow \Omega_{0,\tilde{r}^{\prime}}$ such
that the diagram
\begin{equation}\label{*7}
\begin{diagram}
\node{\tilde X_{0,\tilde r'}}      \arrow[2]{e,t}{f'}
                                    \arrow{se,b}{p'}
\node[2]{X}\arrow{sw,b}{is'}\\
\node[2]{\Omega_{0,\tilde r'}}
\end{diagram}
\end{equation}
is commutative. In particular we have $is'(0)=p'(\tilde 0)$.
Similarly, for every infinite subsequence $\tilde{r} ^{\prime
\prime }$ of $\tilde{r}^{\prime }$ and for every maximal
self-stable
\begin{equation*}
\tilde{X}_{0,\tilde{r}^{\prime \prime }}\supseteq \left\{\tilde{%
x}^{\prime }:\tilde{x}\in \tilde{X}_{0,\tilde{r}^{\prime
}}\right\}
\end{equation*}%
there are a distance-preserving surjection $f^{\prime \prime }:\tilde{X}%
_{0,\tilde{r}^{\prime \prime }}\longrightarrow X$ and an isometry $%
is^{\prime \prime }:X\longrightarrow \Omega _{0,\tilde{r}^{\prime
\prime }}$   with the commutative diagram
$$
\begin{diagram}
\node{\tilde X_{0,\tilde r}} \arrow[2]{e,t}{f''}\arrow{se,b}{p''}
 \node[2]{X}
           \arrow{sw,b}{is''}                      \\
\node[2]{\Omega_{0,\tilde r''}}
\end{diagram}
$$
and with $is''(0)=p''(\tilde 0)$, where $p^{\prime \prime }$ is
the metric identification mapping of the pseudometric space
$\tilde{X}_{0,\tilde{r}^{\prime \prime}}.$ As in the case of
diagram \eqref{4:eq2.26*} we can find an isometric embedding
$\text{em}'':\Omega_{0,\tilde r'}\to\Omega_{0,\tilde r''}$ such
that $\text{em}''\circ p'=p''\circ\text{in\,}r''$ where
$\text{in}''(\tilde x')=(\tilde x')'\in\tilde X_{0,\tilde r''}$
for $\tilde x'\in\tilde X_{0,\tilde r'}$. Repeating the proof of
the commutativity of \eqref{4:eq2.26} we see that the diagram
\begin{equation}\label{8:eq2.45}
\begin{diagram}
\node{\tilde X_{0,\tilde r'}} \arrow[2]{e,t}{\text{in}\,r''}
                                \arrow{se,t}{f'}
                               \arrow[2]{s,l}{p'}
\node[2]{\tilde X_{0,r''}}     \arrow{sw,t}{f''}
                                \arrow[2]{s,r}{p''}\\
\node[2]{X}                     \arrow{sw,t}{is'}
                                \arrow{se,t}{is''}\\
\node{\Omega_{0,\tilde r'}}    \arrow[2]{e,t}{em''}
\node[2]{\Omega_{0,\tilde r''}}
\end{diagram}
\end{equation}
is also commutative. Hence $em^{\prime \prime }$ is surjective because $%
f^{\prime }$ and $is^{\prime \prime }$ are surjections.
Consequently, by Proposition \ref{1:p1.5}, $\Omega
_{0,\tilde{r}^{\prime }}$ is tangent and, as was shown above, for
the isometry $is':X\to \Omega_{0,\tilde r}$ we have
$is'(0)=\alpha$ where $\alpha=p'(\tilde 0)$. Thus, by
definition, $\Omega _{0,\tilde{r}}$ lies in the tangent space $\Omega _{0,%
\tilde{r}^{\prime }}.$

Suppose now that $\Omega_{0,\tilde r}$, the metric identification
of $\tilde X_{0,\tilde r}$, is tangent. To prove the existence of
an isometry
$$
\psi:\Omega_{0,\tilde r}\to X\quad\text{with}\quad \psi(\alpha)=0,
$$
consider, as in the first part of the present proof, a maximal
self-stable family $\tilde X_{0,\tilde r}$ such that inclusion
\eqref{*5} holds and diagram \eqref{*7} is commutative for
function \eqref{*6}. Since $is'$ is an isometry, the commutativity
of \eqref{*7} implies $f'=(is')^{-1}\circ p'$ where $(is')^{-1}$
is the inverse function of $is'$. Then combining the last equality
with \eqref{1:eq1.4} we obtain the commutative diagram
\begin{equation}\label{*8}
\begin{diagram}
\node{\tilde X_{0, \tilde r}} \arrow[2]{e,t}{in_{\tilde r'}}
\arrow{s,l}{p}
                               \node[2]{\tilde X_{0, \tilde r'}}\arrow[2]{e,t}{f'}
                                \arrow{s,r}{p'} \node[2]{X}\\
\node{\Omega_{0, \tilde r}} \arrow[2]{e,t}{em'}
              \node[2]{\Omega_{0, \tilde r'}} \arrow{ene,b}{(is')^{-1}}
\end{diagram}
\end{equation}
Write $ \psi=(is')^{-1}\circ em'. $ For tangent spaces the mapping
$em'$ is an isometry, so $\psi$ is an isometry as a superposition
of two isometries. Note that the commutativity of diagram
\eqref{*8} implies the equality $\psi(\alpha)=0$ for
$\alpha=p(\tilde 0)$.
\end{proof}

\begin{lemma}\label{l:4.6}
Let $Y\subseteq\mathbb C$ be a starlike set with a center $a$ and
let $X:=Con_a(Y)$. Then $X$ and $Y$ are strongly tangent
equivalent at the point $a$.
\end{lemma}
\begin{proof}
If $Y=\{a\}$ then $X=\{a\}$ and this lemma is trivial.
Consequently, we may suppose that $a$ is a limit point of $Y$.

Let us denote by $\hat Y$ the smallest (but not necessarily
closed) cone with the vertex $a$ and such that
$
\hat Y\supseteq Y.
$
Then we have the equality
\begin{equation}\label{*1}
Cl(\hat Y)=X
\end{equation}
where $Cl(\hat Y)$ is the closure of $\hat Y$ in $\mathbb C$.
Indeed, it is easy to prove that $Cl(\hat Y)$ is a closed cone.
Consequently, the inclusion $Cl(\hat Y)\supseteq Con_a(Y)=X$
holds. On the other hand $Con_a(Y)$ is a cone. Hence
$Con_a(Y)\supseteq\hat Y$. It implies
$$
X=Cl(Con_a(Y))\supseteq Cl(\hat Y)
$$
and \eqref{*1} follows.

Write for $t>0$
$$
S_t^Y:=\{y\in Y:|a-y|=t\}\quad\text{and}\quad S_t^X:=\{x\in
X:|a-x|=t\},
$$
i.e., $S_t^Y$ and $S_t^X$ are the spheres in $Y$ and,
respectively, in $X$ with the center $a$ and radius $t$.

Since $Y\subseteq X$ it is sufficient, by Proposition
\ref{6:t5.4}, to prove that
\begin{equation}\label{*2}
\lim_{t\to0}\frac{\varepsilon_a(t,X,Y)}{t}=0
\end{equation}
where
$$
\varepsilon_a(t,X,Y)=\sup_{x\in S_t^X}\inf_{y\in Y}|x-y|.
$$
If equality \eqref{*2} does not hold then there exists
$\delta_0>0$ such that
\begin{equation}\label{*3}
\limsup_{t\to0}\frac{\varepsilon_a(t,X,Y)}t=\delta_0.
\end{equation}
Equality  \eqref{*1} implies that
\begin{equation}\label{*4}
Cl(S_1^{\hat Y})=S_1^X
\end{equation}
where $S_1^{\hat Y}=\{y\in\hat Y:|y-a|=1\}$. Since $S_1^X$ is a
compact subset of $\mathbb C$, it follows from \eqref{*4} that
there is a finite $\frac{\delta_0}2$-net
$\{y_1,\dots,y_n\}\subseteq S_1^{\hat Y}$ for the set $S_1^X$. The
starlikeness of $Y$ implies that there is $\gamma>0$ such that the
implication
$$
(|z-a|<\gamma)\Rightarrow(z\in Y)
$$
is true for every point $z\in\bigcap_{i=1}^nl_a(y_i)$ where
$l_a(y_i)$ are rays starting from $a$ and passing through $y_i$,
see \eqref{8:eq22.34}. Hence for all $t\in\,]0,\gamma[\,$ we have
the inequality
$$
\frac{\varepsilon_a(t,X,Y)}t<\frac{\delta_0}2,
$$
contrary to \eqref{*3}.
\end{proof}

\begin{proof}[Proof of Theorem \ref{t:4.1}]
Without loss of generality we may put $a=0$. Let $\tilde
Y_{a,\tilde r}$ be a maximal self-stable family and let
$\Omega_{a,\tilde r}^Y$ be the corresponding pretangent space.
Write $X=Con_a(Y)$. Then by Lemma \ref{l:4.6} the sets $X$ and $Y$
are strongly tangent equivalent at the point $a$. It follows from
Proposition \ref{8:p2.18} that every pretangent space
$\Omega_{a,\tilde r}^X$ lies in some tangent $\Omega_{a,\tilde
r'}^X$. Consequently, using Proposition \ref{p:2.4}, we have that
$\Omega_{a,\tilde r}^Y$ lies in some tangent space
$\Omega_{a,\tilde r'}^Y$. Suppose now that $\Omega_{a,\tilde r}^X$
is tangent. Write $\tilde X:=[\tilde Y_{a,\tilde r}]_X$. Then
Statement (ii) of Proposition \ref{6:eq5.2} implies that
$\Omega_{a,\tilde r}^X$, the metric identification of $\tilde
X_{a,\tilde r}$, also is tangent. Hence, by Proposition
\ref{8:p2.18}, there is an isometry
$$
\psi^X:\Omega_{a,\tilde r}^X\to X,\qquad \psi^X(\alpha)=a,
$$
with
$$
\tilde X_{a,\tilde r}\ni\tilde
a=(a,\dots,a,\dots)\overset{p}{\to}\alpha\in\Omega_{a,\tilde r}^X.
$$
where $p$ is the projection of $\tilde X_{a,\tilde r}$ on
$\Omega_{a,\tilde r}$. Statement (ii) of Proposition \ref{6:p5.2}
implies that the mapping
$$
\Omega_{a,\tilde
r}^Y\ni\alpha\overset{\nu}{\longmapsto}[\alpha]_X\in\Omega_{a,\tilde
r}^X
$$
is an isometry.  It is easy     to see that the mapping
$\Omega_{a,\tilde r}^Y\overset{\nu}{\to}\Omega_{a,\tilde
r}^X\overset{\psi^X}{\to}X$ is an isometry with the desirable
properties.
\end{proof}
In the next proof we use the notation from the proof of Lemma
\ref{l:4.6}.
\begin{proof}[Proof of Corollary \ref{8:t2.15}]
It follows from the first part of Theorem \ref{t:1.7} that we must
only to prove the equality
\begin{equation}\label{*9}
Con_a(Y)=Conv_a(Y)
\end{equation}
for convex sets $Y\subseteq\mathbb C$ and  $a\in Y$. To prove
this, note that the cone $\hat Y$ is convex for the convex $Y$,
see, for example, \cite[Chapter I, \S 2, Corollary~2.6.3]{Ro} and
that $Cl(\hat Y)$ also is convex \cite[Chapter II, \S6, Theorem
6.1]{Ro}. Moreover, as has been stated in the proof of Lemma
\ref{l:4.6}, the closure of every cone is a cone. Hence
$$
Cl(\hat Y)\supseteq Conv_a (Y).
$$
This inclusion and equality \eqref{*1} imply the inclusion
$Con_a(Y)\supseteq Conv_a(Y)$. Since the reverse inclusion
$Conv_a(Y)\supseteq Con_a(Y)$ is trivial, we obtain \eqref{*9}.
\end{proof}
\begin{proof}[Proof of Theorem \ref{t:1.11}]
Let $Z$ be a starlike set with the center $a$ such that $Z$ and
$X$ are strongly tangent equivalent at the point $a$. To prove the
theorem under consideration we can repeat the proof of Theorem
\ref{t:4.1} using $Z$ instead of $Con_a(Y)$, $X$ instead of $Y$
and Theorem \ref{t:4.1} in place of Proposition \ref{8:p2.18}.
\end{proof}
\begin{lemma}\label{l:4.6*}
Let $X\subseteq\mathbb C$ be a set with a marked point $a$ and let
$l$ be a ray with the vertex $a$. Then we have the equality
\begin{equation}\label{*10}
\lim_{\beta\to0}p(R(X,l,\beta))=\lim_{\beta\to0}p(R(Cl(X),l,\beta)).
\end{equation}
\end{lemma}
\begin{proof}
To prove \eqref{*10} note that
\begin{equation}\label{*11}
p(R(X,l,\beta))\geq p(R(Cl(X),l,\beta))
\end{equation}
because $X\subseteq Cl(X)$. On the other hand, Definition
\ref{9:d1} implies the equality
$$
p(A)=p(Cl(A))
$$
for every $A\subseteq\mathbb R$. Consequently,
\begin{equation}\label{*12}
p(R(X,l,\beta))=p(Cl(R(X,l,\beta))).
\end{equation}
Applying the well-known criterion, see \cite[Proposition
2.1.15]{En}, we see  that a distance function
$$
\mathbb C\ni x\longmapsto|x-a|\in\mathbb R
$$
is closed. The characteristic property of closed maps,
\cite[Chapter 1, \S13, XIV, formula (7)]{Ku} implies
\begin{equation}\label{*13}
Cl(R(X,l,\beta))=\{|z-a|:z\in Cl(X\cap\Gamma(a,l,\beta))\}.
\end{equation}
Moreover, it is easy to see that
\begin{equation}\label{*14}
Cl(X\cap\Gamma(a,l,\beta))\supseteq Cl(X)\cap\Gamma(a,l,\gamma)
\end{equation}
for every $\gamma<\beta$. Relations \eqref{*12}--\eqref{*14} imply
the inequality
$$
p(R(X,l,\beta))\leq p(R(Cl(X),l,\gamma))
$$
for every $\gamma<\beta$. For example we have
$$
p(R(X,l,\beta))\leq p(R(Cl(X),l,\frac\beta{1+\beta})).
$$
Letting $\beta\to0$ in the last inequality and in \eqref{*11} we
obtain \eqref{*10}.
\end{proof}
\begin{lemma}\label{l:4.7}
Let $X\subseteq \mathbb C$ be a set with a marked point $a$, $l$ a
ray with the vertex $a$, $\tilde r=\{r_n\}_{n\in\mathbb N}$ a
normalizing sequence, $\beta_0$ a positive constant and let
$\tilde x,\tilde y$ belong to $\tilde X$, $\tilde
x=\{x_n\}n_{n\in\mathbb N}$, $\tilde y=\{y_n\}_{n\in\mathbb N}$.
Suppose the following conditions are satisfied:
\begin{itemize}
\item[$(i)$] The family $\{\tilde x,\tilde y,\tilde a\}$ is
self-stable w.r.t. $\tilde r$;

\item[$(ii)$] The point $\tilde x$ lies between $\tilde a$ and
$\tilde y$, i.e.,
\begin{equation}\label{*15}
\tilde d_{\tilde r}(\tilde a,\tilde y)=\tilde d_{\tilde r}(\tilde
a,\tilde x)+\tilde d_{\tilde r}(\tilde x,\tilde
y)\quad\text{and}\quad \tilde d_{\tilde r}(\tilde a,\tilde
x)\tilde d_{\tilde r}(\tilde x,\tilde y)\ne0;
\end{equation}

\item[$(iii)$] For every $\beta>0$ there is $n_0\in\mathbb N$ such
that
$
x_n\in\Gamma(a,l,\beta)
$
for all $n\geq n_0$.
\end{itemize}
Then there is $m_0\in\mathbb N$ such that
$y_m\in\Gamma(a,l,\beta_0)
$
for all $m\geq m_0$.
\end{lemma}
\begin{proof}
If the conclusion of the lemma does not hold, then there is a
strictly increasing sequence $\{n_k\}_{k\in\mathbb N}$ of natural
numbers such that
\begin{equation}\label{*16}
y_{n_k}\in\mathbb C\setminus\Gamma(a,l,\beta_0)
\end{equation}
for all $n_k$. Condition $(i)$ implies that
$\{\frac{y_{n_k}}{r_{n_k}}\}_{k\in\mathbb N}$ is a bounded
sequence. Hence, passing if necessary to a subsequence, we may
suppose that $\tilde y'$ is convergent. Condition $(iii)$ and the
existence of $\tilde d(\tilde a,\tilde x)$ imply that
$\{\frac{x_{n_k}}{r_{n_k}}\}_{k\in\mathbb N}$ also is convergent
and
\begin{equation}\label{*17}
x^*:=\lim_{k\to\infty}\frac{x_{n_k}}{r_{n_k}}\in l.
\end{equation}
Write
$$
y^*:=\lim_{k\to\infty}\frac{y_{n_k}}{r_{n_k}}.
$$
It follows from \eqref{*16} that
\begin{equation}\label{*18}
y^*\in\mathbb C\setminus Int(\Gamma(a,l,\beta_0))
\end{equation}
where $Int(\Gamma(a,l,\beta_0))$ is the interior of the sector
$\Gamma(a,l,\beta)$. Using \eqref{*17} and \eqref{*18} it is easy
to show that
$$
|a-y^*|<|a-x^*|+|x^*-y^*|,
$$
see Fig. 2. The last inequality contradicts \eqref{*15} because
$\tilde d_{\tilde r}(\tilde a,\tilde y)=|a-x^*|$, $\tilde
d_{\tilde r}(\tilde a,\tilde y)=|a-y^*|$ and $\tilde d_{\tilde
r}(\tilde x,\tilde y)=|x^*-y^*|$.
\begin{figure}[hbt]\centering
\includegraphics[width=11cm,keepaspectratio]{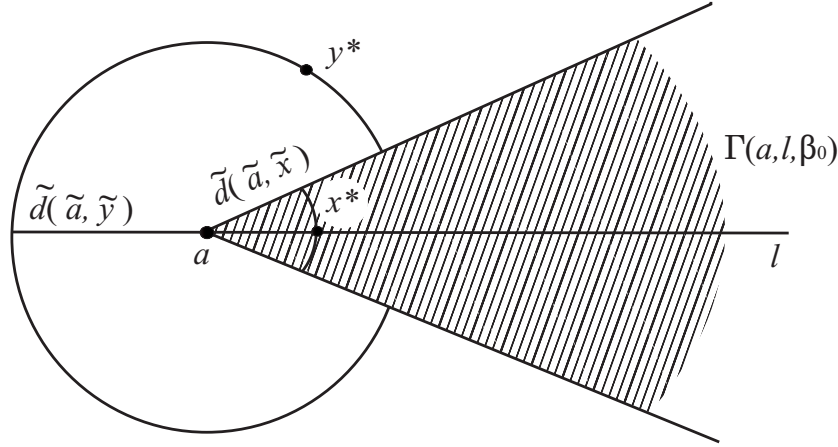}
\caption{The point $x^*$ cannot lie between $a$ and $y^*$.}
\end{figure}
\end{proof}
\begin{proof}[Proof of Theorem \ref{t:1.9}]
Firstly, note that without loss of generality we may assume $X$ to
be  closed. Indeed, by Corollary \ref{c:2.5*}, the sets $X$ and
$Cl(X)$ are strongly tangent equivalent for every $a\in X$, so
using Statement $(ii)$ of Proposition \ref{6:p5.2} we see that for
$X$ and for $Cl(X)$ the supposition of Theorem \ref{t:1.9} is true
(or false) simultaneously. Analogously, using Lemma \ref{l:4.6*}
we can replace $X$ by $Cl(X)$ in the conclusion of Theorem
\ref{t:1.9}.

Suppose there is a ray $l=l_a(b)$ such that
\begin{equation}\label{a10}
1>\lim_{\beta\to0}p(R(X,l,\beta)):=\gamma_0>0.
\end{equation}
Since function $p(R(X,l,\beta))$ is decreasing in $\beta$,
Definition \ref{9:d1} implies that for every $k\in]0,1[$ there is
$\beta_0>0$ with
\begin{equation}\label{a11}
\gamma_0\geq
p_0:=p(R(X,l,\beta_0))=\limsup_{h\to0}\frac{l(0,h,R(X,l,\beta_0))}h>k\gamma_0
\end{equation}
where $l(0,h,R(X,l,\beta_0))$ is the length of the longest
interval in $[0,h]\setminus R(X,l,\beta_0)$. Consequently, there
is a decreasing sequence $\{h_n\}_{n\in\mathbb N},\ h_n>0$,
with\linebreak $\lim_{n\to\infty}h_n=0$ such that
\begin{equation}\label{4.30*}
\lim_{n\to\infty}\frac{l(0,h_n,R(X,l,\beta_0))}{h_n}=p_0.
\end{equation}
Write $]r_n,t_n[$ for the longest open interval in
$[0,h_n]\setminus R(X,l,\beta_0)$. If $r_{m_0}=0$ for some
$m_0\in\mathbb N$, then \eqref{a10} does not hold. Thus we may
suppose $r_m>0$ for all $m\in\mathbb N$. Let
$\{\beta_n\}_{n\in\mathbb N}$ be a decreasing sequence of positive
numbers with $\lim_{n\to\infty}\beta_n=0$. Write
\begin{equation}\label{a12}
\tau_n:=\sup\{d(x,a):x\in B(a,r_n)\cap X\cap\Gamma(a,l,\beta_n)\}
\end{equation}
where $B(a,r_n)=\{x\in X:|x-a|\leq r_n\}$. The previous reasoning
gives the strict inequality
\begin{equation}\label{a13}
\tau_n>0
\end{equation}
for all $n\in\mathbb N$. Since $X$ is closed and $B(a,r_n)$ is
compact, there is  a sequence $\tilde x=\{x_n\}_{n\in\mathbb N}$
such that
\begin{equation}\label{a14}
|x_n-a|=\tau_n
\end{equation}
and
\begin{equation}\label{a15}
x_n\in B(a,\tau_n)\cap X\cap\Gamma(a,l,\beta_n)
\end{equation}
for all $n\in\mathbb N$. Let us obtain now some estimations for
$\lim_{m\to\infty}\frac{r_m}{\tau_m}$ and for
$\lim_{m\to\infty}\frac{t_m}{\tau_m}$. Using \eqref{4.30*} and the
definition of intervals $]r_m,t_m[$ and of the porosity of
$R(X,l,\beta_0)$ we see that
$$
p_0\geq\limsup_{m\to\infty}\frac{|r_m-t_m|}{t_m}\geq
\liminf_{m\to\infty}\frac{|r_m-t_m|}{h_m}=p_0,
$$
that is
\begin{equation}\label{a16}
p_0=\lim_{m\to\infty}\frac{|r_m-t_m|}{t_m}
\end{equation}
and so
\begin{equation}\label{a17}
\lim_{m\to\infty}\frac{r_m}{t_m}=1-p_0.
\end{equation}
Moreover, since the inclusion
$$
R(X,l,\beta_n)\supseteq R(X,l,\beta_{n+1})
$$
holds for all $n\in\mathbb N$, we obtain the inequality
\begin{equation}\label{a18}
\limsup_{m\to\infty}\frac{|\tau_m-t_m|}{t_m}\leq p(R(X,l,\beta_n))
\end{equation}
for all $\mathbb N$. Letting $n\to\infty$ we have
$$
\limsup_{m\to\infty}\frac{|\tau_m-t_m|}{t_m}\leq\gamma_0.
$$
Passing to a subsequence we may suppose that there exists a limit
\begin{equation}\label{a19}
\lim_{m\to\infty}\frac{|\tau_m-t_m|}{t_m}\leq\gamma_0.
\end{equation}
It is clear that
$
|r_m-t_m|\leq|\tau_m-t_m|
$
for all $m\in\mathbb N$. Hence, by \eqref{a11}, \eqref{a16} and by
\eqref{a19},
\begin{equation}\label{a20}
k\gamma_0<\lim_{m\to\infty}\frac{|r_m-t_m|}{t_m}\leq\lim_{m\to\infty}
\frac{|\tau_m-t_m|}{t_m}\leq\gamma_0.
\end{equation}
By the construction we have $\tau_m\leq r_m\leq t_m$. Thus
\eqref{a20} implies
\begin{equation}\label{a21}
k\gamma_0<1-\lim_{m\to\infty}\frac{\tau_m}{t_m}\leq\gamma_0
\end{equation}
or, in  an equivalent form,
\begin{equation}\label{a22}
\frac1{1-k\gamma_0}<\lim_{m\to\infty}\frac{t_m}{\tau_m}\leq\frac1{1-\gamma_0}.
\end{equation}
Similarly we can rewrite \eqref{a21} as
$$
k\gamma_0<\lim_{m\to\infty}\frac{t_m-r_m}{t_m}+\lim_{m\to\infty}\frac{r_m-\tau_m}{t_m}\leq
\gamma_0.
$$
From this, using \eqref{a16}, \eqref{a17}, we obtain
$$
k\gamma_0<p_0+\lim_{m\to\infty}\frac{1-\frac{\tau_m}{r_m}}{\frac1{1-p_0}}\leq\gamma_0
$$
and, after simple calculations,
\begin{equation}\label{a23}
\frac{1-p_0}{1-k\gamma_0}<\lim_{m\to\infty}\frac{r_m}{\tau_m}\leq\frac{1-p_0}{1-\gamma_0}.
\end{equation}
Note that the inequality
\begin{equation}\label{a24}
\frac{1-p_0}{1-\gamma_0}<\frac1{1-k\gamma_0}
\end{equation}
holds if
\begin{equation}\label{a25}
1>k>\frac12(\gamma_0+1).
\end{equation}
Indeed, inequality \eqref{a24} is equivalent to
\begin{equation}\label{a26}
p_0+k\gamma_0-kp_0\gamma_0>\gamma_0.
\end{equation}
By \eqref{a11} we have $p_0>k\gamma_0$. Consequently, to prove
\eqref{a26} it is sufficient to show
$$
2k\gamma_0-kp_0\gamma_0>\gamma_0
$$
that is equivalent to $2k>kp_0+1$ because $\gamma_0>0.$ But
$kp_0<p_0\leq\gamma_0$, see \eqref{a11}, so \eqref{a25} implies
\eqref{a24}.

Let us take the sequence $\tilde \tau=\{\tau_n\}_{n\in\mathbb N}$
as a normalizing sequence. Let $\tilde X_{a,\tilde \tau}$ be a
maximal self-stable family such that $\tilde x\in\tilde
X_{a,\tilde\tau}$ and let $\Omega_{a,\tilde \tau}^X$ be the
corresponding pretangent spaces. The supposition of the theorem
which is being proved, implies that there is a subsequence
$\{n_k\}_{k\in\mathbb N}$ of natural numbers such that
$\Omega_{a,\tilde \tau}^X$ lies in tangent space $\Omega_{a,\tilde
\tau'}^X$, $\tilde\tau'=\{\tau_{n_k}\}_{k\in\mathbb N}$. Replacing
$n$ by $n_k$ in \eqref{a14}, \eqref{a15} we may assume that
$\Omega_{a,\tilde\tau'}^X=\Omega_{a,r}^X$, i.e. $\Omega_{a,\tilde
\tau}^X$ is tangent. Write $\mu:=p(\tilde x)$ where $p$ is the
metric identification mapping of $\tilde X_{a,\tilde \tau}$ and
where $\tilde x=\{x_n\}_{n\in\mathbb N}$ was defined by
\eqref{a14}, \eqref{a15}. It follows from \eqref{a24} and from the
supposition of the theorem that there is $\tilde
y=\{y_n\}_{n\in\mathbb N}\in\tilde X_{a,\tilde \tau}$ with
\begin{equation}\label{a27}
1=\tilde d_{\tilde\tau}(\tilde x,\tilde
a)\leq\frac{1-p_0}{1-\gamma_0} <\tilde d_{\tilde\tau}(\tilde
a,\tilde y)<\frac1{1-k\gamma_0}
\end{equation}
and such that $\mu$ lies between $\alpha:=p(\tilde a)$ and
$\xi:=p(\tilde y)$. Since all conditions of Lemma \ref{l:4.7} are
satisfied by the triple $\tilde x,\tilde y,\tilde a$, there is
$n_0\in\mathbb N$ such that
\begin{equation}\label{a28}
y_n\in\Gamma(a,l,\beta_0)
\end{equation}
for all $n\geq n_0$. Lemma \ref{p:2.6} and relation \eqref{a27}
imply that there is $\varepsilon_0>0$ such that the double
inequality
\begin{equation}\label{a29}
(1+\varepsilon_0)\tau_n\frac{1-p_0}{1-\gamma_0}<d(a,y_n)<\frac{\tau_n}{1-k\gamma}
\end{equation}
holds for all sufficiently large $n$. Moreover, it follows  from
\eqref{a22}, \eqref{a23} that there is $N(\varepsilon)\in\mathbb
N$ such that
\begin{equation}\label{a30}
\frac{\tau_n}{1-k\gamma_0}<t_n<(1+\varepsilon_0)\frac{\tau_n}{1-\gamma_0}
\end{equation}
and
\begin{equation}\label{a31}
\tau_n\frac{1-p_0}{1-k\gamma_0}<r_n<(1+\varepsilon_0)\frac{\tau_n(1-p_0)}{1-\gamma_0}
\end{equation}
for all $n\geq N(\varepsilon)$. The left inequality in \eqref{a29}
and the right one in \eqref{a31} give
$
r_n<d(a,y_n).
$
Similarly, from the right inequality in \eqref{a29} and  from the
left one in \eqref{a30} we obtain
$
d(a,y_n)<t_n.
$
Thus we have
\begin{equation}\label{a32}
d(a,y_n)\in]r_n,t_n[
\end{equation}
for sufficiently large $n$. In addition \eqref{a28} implies
\begin{equation}\label{a33}
d(a,y_n)\in R(X,l,\beta_0).
\end{equation}
 Relations \eqref{a32}, \eqref{a33}
contradict the definition of intervals $]r_n,t_n[$. Hence double
inequality \eqref{a10} does not hold and the theorem follows.
\end{proof}

\begin{example}\label{8:e2.20}
Let $X=\{re^{i\varphi}\in\mathbb C:r\in\mathbb R^+\text{ and
}\varphi\in[\theta_1,\theta_2]\},\ 0<\theta_1-\theta_2\leq\pi\}$
be the closed convex cone, $a=0$ a marked point of $X$ and $\tilde
r=\{r_n\}_{n\in\mathbb N}$ a normalizing sequence. Write for all
$n\in\mathbb N$
\begin{equation}\label{8:eq2.53}
z_n:=\begin{cases} r_ne^{i\theta_1}&\text{if } n \text{ is odd}\\
r_ne^{i\theta_3}&\text{if } n \text{ is even}
\end{cases}
\end{equation}
where a number $\theta_3$ belongs to $]\theta_1,\theta_2[$\,.

Let $\tilde X_{0,\tilde r}\ni\tilde z=\{z_n\}_{n\in\mathbb N}$ be
a maximal self-stable family with the corresponding pretangent
space $\Omega_{0,\tilde r}$. We claim that $\Omega_{0,\tilde r}$
is not tangent.

Indeed, suppose that $\Omega_{0,\tilde r}$ is tangent. Write
$\tilde r^{(1)}:=\{r_{2n+1}\}_{n\in\mathbb N}$ and $\tilde
r^{(2)}:=\{r_{2n}\}_{n\in\mathbb N}$. By Statement $(ii)$ of
Proposition \ref{1:p1.5} the families
$$
\tilde X_{0,\tilde r^{(1)}}:=\{\{x_{2n+1}\}_{n\in\mathbb
N}:\{x_n\}_{n\in\mathbb N}\in\tilde X_{0,\tilde r}\} $$ and $$
\tilde X_{0,\tilde r^{(2)}}:=\{\{x_{2n}\}_{n\in\mathbb
N}:\{x_n\}_{n\in\mathbb N}\in\tilde X_{0,\tilde r}\}
$$
are also maximal self-stable and corresponding spaces
$\Omega_{0,\tilde r^{(1)}}$, $\Omega_{0,\tilde r^{(2)}}$ are
tangent. Passing, if necessary, to subsequences we may suppose
that for every $\tilde x=\{x_n\}_{n\in\mathbb N}\in\tilde
X_{0,\tilde r}$ there are limits
\begin{equation}\label{8:eq2.54}
f^{(1)}(\tilde
x):=\lim_{n\to\infty}\frac{x_{2n+1}}{r_{2n+1}}\quad\text{and}\quad
f^{(2)}(\tilde x):=\lim_{n\to\infty}\frac{x_{2n}}{r_{2n}}.
\end{equation}
Let $\lfloor\tilde X_{0,\tilde r}\rfloor$ be a system of distinct
representatives of the factor space $\Omega_{0,\tilde r}$, i.e.,
for every $\tilde x\in\tilde X_{0,\tilde r}$ there is a unique
$\tilde y\in\lfloor\tilde X_{0,\tilde r}\rfloor$ such that $\tilde
d(\tilde x,\tilde y)=0$. Then, by Statement $(iii)$ of Lemma
\ref{8:l2.16}, the mappings
$$
\lfloor\tilde X_{0,\tilde r}\rfloor\ni\tilde x\longmapsto
f^{(1)}(\tilde x)\in X\quad \text{and}\quad \lfloor\tilde
X_{0,\tilde r}\rfloor\ni\tilde x\longmapsto f^{(2)}(\tilde x)\in X
$$
are isometries, where $f^{(1)}(\tilde x)$ and $f^{(2)}(\tilde x)$
are defined by \eqref{8:eq2.54}. Therefore, we have
$$
f^{(1)}(\tilde z)=e^{i\theta_1}\quad \text{and}\quad
f^{(2)}(\tilde z)=e^{i\theta_3}.
$$
 Hence, there is an isometry $\psi:X\to X$ such
that
\begin{equation}\label{8:eq2.55}
\psi(e^{i\theta_3})=e^{i\theta_1}.
\end{equation}
Note that $e^{i\theta_3}\in Int \,X$ and $e^{i\theta_1}\in\partial
X$. Consequently, equality \eqref{8:eq2.55} contradicts the
Brouwer Theorem on the invariance of the open sets.
\end{example}
This example  and Statement $(ii)$ of Proposition \ref{6:p5.2}
imply Proposition~\ref{p:1.12}.

\bigskip{\bf Acknowledgement.} The results
of this paper were partly produced during the visit of the first
author to the Mersin University (TURKEY) in February-April 2008
under the support of the T\"{U}B\.{I}TAK-Fellowships For Visiting
Scientists Programme.

\end{document}